\documentclass{amsart}
\usepackage[T1]{fontenc}
\usepackage[utf8]{inputenc}
\usepackage[english]{babel}
\usepackage{amssymb,amsmath}
\usepackage{suffix}
\usepackage{graphicx}
\usepackage{multirow}
\usepackage{subcaption}
\usepackage{geometry}
\usepackage{hyperref}
\hypersetup{
  linkcolor=blue,
  citecolor=blue,
  filecolor=blue,
  urlcolor=MidnightBlue,
  pdftitle={High order variational integrators in the optimal control of mechanical systems},
  pdfauthor={Cedric M. Campos, Sina Ober-Bloebaum and Emmanuel Trelat},
  pdfsubject={submitted to Discrete and Continuous Dynamical Systems - Series A},
  pdfkeywords={optimal control, mechanical systems, geometric integration, variational integrator, high order, Runge-Kutta, direct methods, commutation property},
  pdfcreator={GNU Emacs 24.4.1},
}

\newcommand{\bbR}{\mathbb{R}}
\newcommand{\calC}{\mathcal{C}}
\newcommand{\calH}{\mathcal{H}}
\newcommand{\calP}{\mathcal{P}}
\newcommand{\calQ}{\mathcal{Q}}
\newcommand{\caldQ}{\dot{\mathcal{Q}}}
\newcommand{\calU}{\mathcal{U}}
\newcommand{\ext}{\mathop{\mathrm{ext}}}

\newcommand{\diff}{\mathrm{d}}
\newcommand{\dt}{\diff t}
\newcommand{\dd}[2][{}]{\frac{\diff#1}{\diff#2}}
\newcommand{\ddt}{\dd[]t}
\newcommand{\pp}[2][{}]{\frac{\partial#1}{\partial#2}}
\WithSuffix\newcommand\pp*[2][{}]{\partial#1/\partial#2}
\newcommand{\dq}{\dot{q}}
\newcommand{\dQ}{\dot{Q}}
\newcommand{\dP}{\dot{P}}
\newcommand{\leg}{\mathop{\mathrm{leg}}\nolimits}
\newcommand{\cf}{\emph{cf.}}
\newcommand{\eg}{\emph{e.g.}}

\newcommand{\ie}{\emph{i.e.}}

\newcommand{\quand}{\quad\text{and}\quad}
\newcommand{\qquand}{\qquad\text{and}\qquad}

\newtheorem{theorem}{Theorem}[section]

\newtheorem{proposition}[theorem]{Proposition}

\theoremstyle{definition}

\theoremstyle{remark}
\newtheorem{remark}[theorem]{Remark}
\newtheorem{problem}[theorem]{Problem}
\newtheorem{example}[theorem]{Example}

\title[HOVI in optimal control]{High order variational integrators in the optimal control of mechanical systems}
\author[C. M. Campos]{Cédric M. Campos}
\address{IMUVA, Universidad de Valladolid, 47011 Valladolid, Spain; Instituto de Ciencias Matemáticas, 28049 Madrid, Spain}
\email{cedricmc@\{uva,icmat\}.es}
\author[S. Ober-Blöbaum]{Sina Ober-Blöbaum}
\address{Department of Mathematics, University of Paderborn, 33098 Paderborn, Germany}
\email{sinaob@math.upb.de}
\author[E. Trélat]{Emmanuel Trélat}
\address{Sorbonne Universités, UPMC Univ. Paris 06, CNRS UMR 7598, Laboratoire; Jacques-Louis Lions, Institut Universitaire de France, F-75005, Paris, France}
\email{emmanuel.trelat@upmc.fr}

\begin{document}
\begin{abstract}
In recent years, much effort in designing numerical methods for the simulation and optimization of mechanical systems has been put into schemes which are structure preserving. One particular class are variational integrators which are momentum preserving and symplectic. In this article, we develop two high order variational integrators which distinguish themselves in the dimension of the underling space of approximation and we investigate their application to finite-dimensional optimal control problems posed with mechanical systems. The convergence of state and control variables of the approximated problem is shown. Furthermore, by analyzing the adjoint systems of the optimal control problem and its discretized counterpart, we prove that, for these particular integrators, dualization and discretization commute.
\end{abstract}
\maketitle

\section{Introduction}\label{sec:intro}
In practice, solving an optimal control problem requires the a priori choice of a numerical method. Many approaches do exist, that are either based on a direct discretization of the optimal control problem, resulting into a nonlinear programming problem, or based on the application of the Pontryagin Maximum Principle, reducing into a boundary value problem. The first class of approaches are called direct, and the second ones, based on the preliminary use of the Pontryagin maximum principle, are called indirect. It can be noted that indirect methods, although very precise, suffer from an important drawback: Unless one already has a good knowledge of the optimal solution, they are very difficult to initialize since they are extremely sensitive to the initial guess. Although many solutions exist in order to carefully initialize a shooting method (see \cite{Tr05,Tr12}), in most of engineering applications direct approaches are preferred due to their simplicity and robustness. Roughly speaking, direct methods consist of
\begin{enumerate}
\item discretizing first the cost and the differential system, in order to reduce the optimal control problem to a usual nonlinear minimization problem with constraints, also called nonlinear programming problem (NLP), with dimension inversely proportional to the smoothness of the discretization;
\item and then dualizing, by applying for instance a Lagrange-Newton method to the NLP, deriving the Karush-Kuhn-Tucker equations (KKT), also called discrete adjoint system, and applying a Newton method to solve the resulting optimality system.
\end{enumerate}
Many variants exist, \eg\  \cite{Betts98}).
In contrast,  indirect methods consist of
\begin{enumerate}
\item first dualizing the optimal control problem to derive the adjoint system, by applying the Pontryagin Maximum Principle (PMP) (or, equivalently, the Lagrange multipliers necessary condition for optimality in infinite dimension),
\item and then discretizing, by applying a shooting method (that is, a Newton method composed with a numerical integration method).
\end{enumerate}
In shorter words, direct methods consist of 1) discretize, 2) dualize, and indirect methods consist of the converse: 1) dualize, 2) discretize. It is natural to wonder whether this diagram is commutative or not, under usual approximation assumptions.

Since the pioneering works of \cite{Ha00, Ro05}, it is well known by now that, in spite of usual assumptions ensuring consistency and stability, direct methods may diverge. In other words discretization and dualization do not commute in general. This is due to a complex interplay with the mesh, to the appearance of spurious highfrequencies ultimately causing the divergence (see \cite{Ha00} for very simple finite-dimensional linear quadratic problems and \cite{Zu05} for infinite-dimensional wave equations).

Several remedies and solutions exist, from which \cite{Bo06,GoRoKaFa08,Ha00,RoFa03} are a representative sample. For instance, the results of \cite{Ha00} assert the convergence of optimal control problems under specific smoothness and coercivity assumptions provided that the underlying discretization method be based on a Runge-Kutta method. However, the convergence order of the optimal control solution, which depends on the convergence order of the state and the resulting adjoint scheme, is reduced compared to the order of the Runge-Kutta method applied to the state system. Indeed, the discrete state and adjoint system constitutes a symplectic partitioned Runge-Kutta method for which order conditions on the Runge-Kutta coefficients are derived to preserve the convergence rates. Whereas in \cite{Ha00} the symplectic partitioned Runge-Kutta scheme for state and adjoint is explicitly derived, in a recent overview article \cite{SS14} a proof is given based on a relation between quadratic invariants (that are also naturally preserved by symplectic partitioned Runge-Kutta integrators) and the fulfillment of the KKT equations. The preservation of convergence rates is referred to as the Covector Mapping Principle (CMP) (see \eg\ \cite{GoRoKaFa08}). More precisely, the CMP is satisfied if there exists an order-preserving map between the adjoint variables corresponding to the dualized discrete problem (KKT) and the discretized dual problem (discretized PMP). For the class of Legendre pseudospectral methods the CMP is proven if additional closure conditions are satisfied (see \cite{GoRoKaFa08, RoFa03}), whereas for Runge-Kutta methods the CMP holds if the order conditions on the Runge-Kutta coefficients derived in \cite{Ha00} are satisfied.
For a detailed discussion on the commutation properties we refer to \cite{Ro05}.

While for general dynamical systems, many studies of discretizations of optimal control problems are based on Runge-Kutta methods (see \eg\ \cite{DoHa01,DoHa00,Ha01,HPS13,Wa07}), particularly for mechanical systems, much effort in designing numerical methods for the simulation and optimization of such systems has been put into schemes which are structure preserving in the sense that important qualitative features of the original dynamics are preserved in its discretization (for an overview on structure preserving integration methods see \eg\ \cite{HaLuWa02}). One special class of structure preserving integrators is the class of variational integrators, introduced in \cite{MaWe01,Su90}, which are symplectic and momentum-preserving and have an excellent long-time energy behavior. Variational integrators are based on a discrete variational formulation of the underlying system, \eg\ based on a discrete version of Hamilton's principle or of the Lagrange-d'Alembert principle for conservative \cite{LMOW04,LMOWe04} or dissipative \cite{Kane00} mechanical systems, respectively. They have been further extended to different kind of systems and applications, \eg\ towards constrained \cite{CoMa01,KoMaSu10,leyendecker07-2}, non smooth \cite{FMOW03}, multirate and multiscale \cite{LO12,StGr09,TaOwMa2010}, Lagrangian PDE systems \cite{Lew03, MPS98} and electric circuits \cite{OTCOM13}. In the cited works, typically quadrature rules of first or second order are used in order to approximate the action functional of the system. To design high order variational integrators, higher order quadrature rules based on polynomial collocation can be employed. Such so called Galerkin variational integrators have been introduced in \cite{MaWe01} and further studied in \cite{Ca13,HaLe13,Leok2011,O14,OS14}.

For numerically solving optimal control problems by means of a direct method, the use of variational integrators for the discretization of the problem has been proposed in \cite{CaJuOb12,DMOCC,ObJuMa10}. This approach, denoted by DMOC (Discrete Mechanics and Optimal Control), yields a finite-difference type discretization of the dynamical constraints of the problem which by construction preserves important structural properties of the system, like the evolution of momentum maps associated to symmetries of the Lagrangian or the energy behavior. For one class of Galerkin variational integrators, that is equivalent to the class of symplectic partitioned Runge-Kutta methods, the adjoint system and its convergence rates are analyzed in \cite{ObJuMa10}. It is shown that, in contrast to a discretization based on standard Runge-Kutta methods in \cite{Ha00}, the convergence order of the discrete adjoint system is the same as for the state system due to the symplecticity of the discretization scheme. In particular, we obtain the same symplectic-momentum scheme for both state and adjoint system, that means that discretization and dualization commute for this class of symplectic schemes and the CMP is satisfied.
For general classes of variational integrators, the commutation property is still an open question.
The contribution of this work is twofold:
\begin{enumerate}
\item We derive two different kinds of high order variational integrators based on different dimensions of the underling polynomial approximation (Section~\ref{sec:discrete}).
While the first well-known integrator is equivalent to a symplectic partitioned Runge-Kutta method, the second integrator, denoted as symplectic Galerkin integrator, yields a ``new'' method which in general, cannot be written as a standard symplectic Runge-Kutta scheme.
\item For the application of high order variational integrators to finite-dimensional optimal control problems posed with mechanical systems, we show the convergence of state and control variables and prove the commutation of discretization and dualization (Sections \ref{sec:conv} and \ref{sec:conmutation}).
\end{enumerate}
The paper is structured as follows: In Section~\ref{sec:ocmech} the optimal control problem for a mechanical system is introduced. Its discrete version is formulated in Section~\ref{sec:discrete} based on the derivation of two different kinds of high order variational integrators. The first main result is stated in Section~\ref{sec:conv}: Under specific assumptions on the problem setting we prove the convergence of the primal variables for an appropriate choice of the discrete controls (Theorem~\ref{thm:control:conv}). Along the lines of \cite{Ha01}, we demonstrate the influence of the control discretization on the convergence behavior by means of several examples. In Section~\ref{sec:conmutation} the discrete adjoint system for the symplectic Galerkin method is derived. It turns out that the reformulation of the variational scheme in Section~\ref{sec:discrete} simplifies the analysis. Whereas commutation of discretization and dualization for symplectic partitioned Runge-Kutta methods has already been shown in \cite{ObJuMa10}, we prove the same commutation property for the symplectic Galerkin discretization (Theorem~\ref{th:commutation}), which is the second main result of this work. In contrast to the discretization with Legendre pseudospectral methods or classical Runge-Kutta methods, no additional closure conditions (see \cite{GoRoKaFa08}) or conditions on the Runge-Kutta coefficients (see \cite{Ha00}) are required to satisfy the CMP, respectively. Furthermore, using the high order variational integrators presented here, not only the order but also the discretization scheme itself is preserved, \ie\ one yields the same schemes for the state and the adjoint system. We conclude with a summary of the results and an outlook for future work in Section~\ref{sec:conclusion}.

\section{Optimal control for mechanical systems}\label{sec:ocmech}

\subsection{Lagrangian dynamics}\label{subsec:Lag}
One of the main subjects of Geometric Mechanics is the study of dynamical systems governed by a Lagrangian. Typically, one considers a mechanical system with \emph{configuration manifold} $Q\subseteq \bbR^n$ together with a \emph{Lagrangian function} $L\colon TQ\to\bbR$, where the associated \emph{state space} $TQ$ describes the position $q$ and velocity $\dq$ of a particle moving in the system. Usually, the Lagrangian takes the form of kinetic minus potential energy, $ L(q,\dq) = K(q,\dq) - V(q) = \frac12\,\dq^T\cdot M(q)\cdot\dq - V(q)$, for some (positive definite) \emph{mass matrix} $M(q)$.

A consequence of the \emph{principle of least action}, also known as \emph{Hamilton's principle}, establishes that the natural motions $q\colon[0,T]\to Q$ of the system are characterized by stationary solutions of the \emph{action}, thus, motions satisfying
\begin{equation}\label{eq:principle.hamilton}
\delta \int_0^T L(q(t),\dq(t))\,\dt = 0
\end{equation}
for zero initial and final variations $\delta q(0)=\delta q(T) =0$. The resulting equations of motion are the \emph{Euler-Lagrange equations} (refer to \cite{AbMa78}),
\begin{equation} \label{eq:EL-equation}
\ddt\pp[L]{\dq} - \pp[L]{q} = 0 \,.
\end{equation}
When the Lagrangian is \emph{regular}, that is when the velocity Hessian matrix $\pp*[^2L]{\dq^2}$ is non-degenerate, the Lagrangian induces a well defined map, the \emph{Lagrangian flow}, $F_L^t\colon TQ\to TQ$ by $F_L^t(q_0,\dq_0) := (q(t),\dq(t))$, where $q\in\calC^2([0,T],Q)$ is the unique solution of the Euler-Lagrange equation \eqref{eq:EL-equation} with initial condition $(q_0,\dq_0)\in TQ$. By means of the \emph{Legendre transform} $\leg_L:(q,\dq)\in TQ\mapsto (q,p=\pp[L]{\dq}|_{(q,\dq)})\in T^*Q$, where $T^*Q$ is the \emph{phase space} of positions $q$ and momenta $p$, one may transform the Lagrangian flow into the \emph{Hamiltonian flow} $F_H^t\colon T^*Q\to T^*Q$ defined by $F_H^t(q_0,p_0) := \leg_L(q(t),\dq(t))$.

Moreover, different preservation laws are present in these systems. For instance the Hamiltonian flow preserves the natural symplectic structure of $T^*Q$ and the total energy of the system, typically $H(q,p) = K(q,p) + V(q) = \frac12\,p^T\cdot M(q)^{-1}\cdot p - V(q)$ (here $K$ still denotes the kinetic energy, but depending on $p$ rather than on $\dq$). Also, if the Lagrangian possess Lie group symmetries, then \emph{Noether's theorem} asserts that the associated momentum maps are conserved, like for instance the linear momentum and/or the angular momentum.

If external (non conservative) forces $F\colon (q,\dq)\in TQ \mapsto (q,F(q,\dq))\in T^*Q$ are present in the system, Hamilton's principle~\eqref{eq:principle.hamilton} is replaced by the \emph{Lagrange-d'Alembert principle} seeking for curves that satisfy the relation
\begin{equation} \label{eq:principle.dalembert}
\delta \int_0^T L(q,\dq) \,\dt +  \int_0^T F(q,\dq) \cdot \delta q \,\dt = 0
\end{equation}
for zero boundary variations $\delta q(0)=\delta q(T) =0$, where the second term is denoted as \emph{virtual work}. This principle leads to the \emph{forced Euler-Lagrange equations}
\begin{equation} \label{eq:ELF-equation}
 \ddt\pp[L]{\dq}-\pp[L]{q} = F(q,\dq) \,.
\end{equation}
The forced version of Noether's theorem (see \eg\ \cite{MaWe01}) states that if the force acts orthogonal to the symmetry action, then momentum maps are still preserved by the flow. Otherwise, the change in momentum maps and energy is determined by the amount of forcing in the system.

\subsection{Optimal control problem}\label{subsec:ocp}

Since we are interested in optimally controlling Lagrangian systems, we assume that the mechanical system may be driven by means of some time dependent control parameter $u:[0,T]\rightarrow U$ with $U \subset \bbR^m$ being the \emph{control set}. Typically, the control appears as an extra variable in the external force such that in the following we consider forces of the form $F: TQ \times U \rightarrow T^*Q$ and we replace the right-hand side of \eqref{eq:ELF-equation} by the control dependent force term $F(q,\dq,u)$.

An optimal control problem for a mechanical system reads (also denoted as Lagranigan optimal control problem in \cite{ObJuMa10})

\begin{problem}[Lagrangian optimal control problem (LOCP)]\label{prob:locp}
\begin{subequations}\label{eq:LOCP}
\begin{equation}\label{eq:LOCP:cost}
\min_{q, \dq, u} J(q,\dq,u) = \int_0^T C(q(t),\dq(t),u(t))\,\dt  + \Phi(q(T),\dq(T))
\end{equation}
subject to \begin{eqnarray}
\delta \int_0^T L(q(t),\dq(t)) \, \dt + \int_0^T F(q(t),\dq(t),u(t))\cdot \delta q(t) \, \dt &=& 0, \label{eq:LOCP_dA}\\
\label{eq:LOCP3} (q(0),\dq(0)) &=& (q^0,\dq^0),\label{eq:LOCP_ini}
\end{eqnarray}
\end{subequations}
with minimization over $q\in C^{1,1}([0,T],Q)= W^{2,\infty}([0,T],Q)$, $\dq \in W^{1,\infty}([0,T],T_qQ)$ and $u\in L^\infty([0,T],U)$.
The interval length $T$ may either be fixed, or appear as degree of freedom in the optimization problem. Since any optimal control problem with free final time can easily be transformed into a problem with fixed final time (see \eg\ \cite{Gerdts12}), we assume the time $T$ to be fixed from now on. The control set $U\subset\bbR^m$ is assumed to be closed and convex, and the \emph{density cost function} $C\colon TQ\times U\mapsto \bbR$ and the \emph{final cost function} $\Phi\colon TQ \mapsto \bbR^n$ are continuously differentiable, being $\Phi$ moreover bounded from below.
\end{problem}

Henceforth we should assume that the Lagrangian is regular, \ie\ there is a (local) one-to-one correspondence between the velocity $\dq$ and the momentum $p$ via the Legendre transform and its inverse
\[
p = \pp[L]{\dq}(q,\dq) \quand \dq = \left(\pp[L]{\dq}\right)^{-1}(q,p).
\]
Thus, the forced Euler-Lagrange equations \eqref{eq:ELF-equation} can be transformed into the partitioned system
\begin{equation}\label{eq:partsys}
\dq(t)= f(q(t),p(t))\,, \quad \dot{p}(t)= g(q(t),p(t),u(t))
\end{equation}
with
\begin{equation}\label{eq:fg}
f(q,p) =  \left(\pp[L]{\dq}\right)^{-1}(q,p)\quand g(q,p,u)=\pp[L]q\left(q, f(q,p)\right) + F\left(q,f(q,p),u\right).
\end{equation}
With some abuse of notation we denote the force and the cost functions defined on $T^*Q\times U$ and $T^*Q$, respectively, by $F(q,p,u):=F(q,f(q,p),u)$, $C(q,p,u):=C(q,f(q,p),u)$ and $\Phi(q,p):=\Phi(q,f(q,p))$ such that Problem~\ref{prob:locp} can be formulated as an optimal control problem for the partitioned system \eqref{eq:partsys}.

\begin{problem}[Optimal control problem (OCP)]\label{prob:ocp}
\begin{subequations}\label{eq:OCP}
\begin{equation}\label{eq:OCP:cost}
\min_{q, p, u} J(q,p,u) = \int_0^T C(q(t),p(t),u(t))\,\dt  + \Phi(q(T),p(T))
\end{equation}
subject to \begin{align}
\dq(t)&= f(q(t),p(t))\,, \quad q(0)=q^0\,, \label{eq:OCP2}\\
\dot{p}(t)&= g(q(t),p(t),u(t))\,, \quad p(0)=p^0, \label{eq:OCP3}
\end{align}
\end{subequations}
with minimization over $q\in  W^{2,\infty}([0,T],Q)$, $p \in W^{1,\infty}([0,T],T_q^*Q)$ and $u\in L^\infty([0,T],U)$ and the functions $f\colon T^*Q \mapsto \bbR^n$, $g\colon T^*Q \times U \mapsto \bbR^n$ are assumed to be Lipschitz continuous.
\end{problem}

The first order necessary optimality conditions can be derived by means of the Hamiltonian for the optimal control problem given by
\begin{equation}
\calH(q,p,u,\lambda,\psi,\rho_0) = \rho_0 C(q,p,u) + \lambda\cdot f(q,p) + \psi\cdot g(q,p,u)
\end{equation}
with $\rho_0 \in \bbR$ and $\lambda$ and $\psi$ are covectors in $\bbR^n$.

\begin{theorem}[Minimum Principle, \eg\  \cite{Gerdts12}]\label{th:pontryagin}
Let $(q^*,p^*,u^*)\in W^{2,\infty}([0,T],\allowbreak Q) \times W^{1,\infty}([0,T],T^*_{q^*}Q) \times L^{\infty}([0,T],U)$ be an optimal solution to Problem~\ref{prob:ocp}. Then there exist functions $\lambda \in W^{1,\infty}([0,T],\bbR^n)$ and $\psi \in W^{1,\infty}([0,T],\bbR^n)$ and
a constant $\rho_0\ge0$ satisfying $(\rho_0,\lambda,\psi)\not=(0,0,0)$ for all $t\in [0,T]$ such that
\begin{subequations}\label{eq:minprinc}
\begin{equation}\label{eq:min_ham}
\calH(q^*(t),p^*(t),u^*(t),\lambda(t),\psi(t),\rho_0) = \min\limits_{u\in U} \calH(q(t),p(t), u,\lambda(t),\psi(t),\rho_0)\,,
\end{equation}
for $t\in[0,T]$, and $(\rho_0,\lambda,\psi)$ solves the following initial value problem:
\begin{align}
\dot{\lambda}& = -\nabla_q\calH(q^*,p^*,u^*,\lambda,\psi,\rho_0), &\lambda(T) =\rho_0 \nabla_q\Phi(q^*(T),p^*(T)),\label{eq:pontq}\\
\dot{\psi}& = -\nabla_p\calH(q^*,p^*,u^*,\lambda,\psi,\rho_0), &\psi(T) = \rho_0 \nabla_p\Phi(q^*(T),p^*(T)).\label{eq:pontp}
\end{align}
\end{subequations}
\end{theorem}

The vectors $\lambda(t)$ and $\psi(t)$ are the \emph{costate} or the \emph{adjoint variables} of the Hamiltonian equations of optimal control. The scalar $\rho_0$ is called the \emph{abnormal multiplier}. In the abnormal case, it holds $\rho_0=0$, and otherwise the multiplier can be normalized to $\rho_0=1$. Since no final constraint on the state is present in the optimal control problem, the above principle holds true with $\rho_0=1$ (as proved for instance in \cite{Tr05}).

\begin{remark}
If $g$ is affine w.r.t. $u$, then the topologies can be taken as $L^2$ for the controls, $H^2=W^{2,2}$ on $q$ and $H^1=W^{1,2}$ on $p$, and the PMP would still be valid for these classes. Besides, optimal control problems where the optimal control is in $L^2$ but not in $L^\infty$ are very seldom. For instance, if one is able to express $u$ in function of $(q,p,\lambda,\psi)$, as for the assumptions of Theorem \ref{th:commutation}, then $u$ is clearly in $L^\infty$.
\end{remark}

\section{Discretization}\label{sec:discrete}
Since we are interested in solving optimal control problems by some kind of direct method, a discrete approximation of Problem~\ref{prob:ocp} is required. To this end, we first introduce two different variational integrators that we employ for the approximation of the control system given in \eqref{eq:OCP2}-\eqref{eq:OCP3}. Based on these discrete schemes, we derive the discrete approximations of the optimal control problem that can be solved by standard numerical optimization methods. The controls play no role in the derivations of the variational integrators, therefore we will omit temporarily the dependence of the external force $F$ on $u$, which will lighten the notation. The discrete schemes including the approximation of the controls are given in Section~\ref{subsec:docp}.

\subsection{Discrete Mechanics and Variational Integrators}\label{sec:dm}
Discrete Mechanics is, roughly speaking, a discretization of Geometric Mechanics theory. As a result, one obtains a set of discrete equations corresponding to the Euler-Lagrange equation \eqref{eq:ELF-equation} above but, instead of a direct discretization of the ODE, the latter are derived from a discretization of the base objects of the theory, the state space $TQ$, the Lagrangian $L$, etc. In fact, one seeks for a sequence $\{(t_0,q_0),(t_1,q_1),\dots,\allowbreak(t_N,q_N)\}$ that approximates the actual trajectory $q(t)$ of the system ($q_k\approx q(t_k)$), for a constant time-step $h=t_{k+1}-t_k>0$.

A \emph{variational integrator} is an iterative rule that outputs this sequence and it is derived in an analogous manner to the continuous framework. Given a discrete Lagrangian $L_d\colon Q\times Q\to\bbR$ and discrete forces $F^\pm_d\colon Q\times Q\to T^*Q$, which are in principle thought to approximate the continuous Lagrangian action and the virtual work, respectively, over a short time
\begin{subequations}
\begin{align}
 L_d(q_k,q_{k+1}) &\approx \int_{t_k}^{t_{k+1}}L(q(t),\dq(t))\dt \,,\\
 F_d^-(q_k,q_{k+1})\cdot \delta q_k + F_d^+(q_k,q_{k+1})\cdot \delta q_{k+1} &\approx \int_{t_k}^{t_{k+1}}F(q(t),\dq(t))\cdot \delta q(t) \dt \,,
 \end{align}
 \end{subequations}
one applies a variational principle to derive the well-known forced discrete Euler-Lagrange (DEL) equation,
\begin{equation} \label{eq:DEL}
D_1L_d(q_k,q_{k+1}) + D_2L_d(q_{k-1},q_k) + F_d^-(q_k,q_{k+1}) + F_d^+(q_{k-1},q_k) = 0 \,,
\end{equation}
for $k=1,\ldots,N-1$, where $D_i$ stands for the partial derivative with respect to the $i$-th component. The equation defines an integration rule of the type $(q_{k-1},q_k) \mapsto (q_k,q_{k+1})$, however if we define the pre- and post-momenta (also denoted as \emph{discrete Legendre transforms})
\begin{subequations}\label{eq:discrete.legendre.transform}
\begin{align}
p_k^- &:= -D_1L_d(q_k,q_{k+1})- F^-_d(q_k,q_{k+1}),\quad  k=0,\ldots,N-1,\quad\text{and}\\
 p_k^+& := D_2L_d(q_{k-1},q_k) + F^+_d(q_{k-1},q_k), \quad  k=1,\ldots,N,
 \end{align}
\end{subequations}
the discrete Euler-Lagrange equation \eqref{eq:DEL} is read as the momentum matching $p_k^-=p_k^+=:p_k$ and defines an integration rule of the type $(q_k,p_k)\mapsto(q_{k+1},p_{k+1})$.

The nice part of the story is that the integrators derived in this way naturally preserve (or nearly preserve) the quantities that are preserved in the continuous framework, the symplectic form, the total energy (for conservative systems) and, in presence of symmetries, the linear and/or angular momentum (for more details, see \cite{MaWe01}). Furthermore, other aspects of the continuous theory can be ``easily'' adapted, symmetry reduction \cite{CaCeDi12,CoJiMa12,IgMaMa08}, constraints \cite{JoMu09,KoMaSu10}, control forces \cite{CaJuOb12,ObJuMa10}, etc.

\subsection{High Order Variational Integrators}\label{sec:hovi}
High order variational integrators for time dependent or independent systems (HOVI[t]) are a class of integrators that, by using a multi-stage approach, aim at a high order accuracy on the computation of the natural trajectories of a mechanical system while preserving some intrinsic properties of such systems. In particular, symplectic-partitioned Runge-Kutta methods (spRK) and, what we call here, symplectic Galerkin methods (sG) are $s$-stage variational integrators of order up to $2s$.

The derivation of these methods follows the general scheme that comes next, the specifics of each particular case are detailed in the following subsections. For a fixed time step $h$, one considers a series of points $q_k$, refereed as macro-nodes. Between each couple of macro-nodes $(q_k,q_{k+1})$, one also considers a set of micro-data, the $s$ stages: For the particular cases of sG and spRK methods, we consider micro-nodes $Q_1,\dots,Q_s$ and micro-velocities $\dQ_1,\dots,\dQ_s$, respectively. Both macro-nodes and micro-data (micro-nodes or micro-velocities) are required to satisfy a variational principle, giving rise to a set of equations, which properly combined, define the final integrator.

Here and after, we will use the following notation: Let $0\leq c_1<\ldots<c_s\leq1$ denote a set of collocation points and consider the associated Lagrange polynomials and nodal weights, that is,
\[ l^j(t) := \prod_{i\neq j}\frac{t-c_i}{c_j-c_i} \qquand b_j := \int_0^1l^j(t)\dt \,, \]
respectively. Note that the pair of $(c_i,b_i)$'s define a quadrature rule and that, for appropriate $c_i$'s, this rule may be a Gaussian-like quadrature, for instance, Gauss-Legendre, Gauss-Lobatto, Radau or Chebyshev.

Now, for the sake of simplicity and independently of the method, we will use the same notation for the nodal coefficients. We define for spRK and sG, respectively,
\begin{equation}\label{eq:coeff_ab}
 a_{ij} := \int_0^{c_i}l^j(t)\dt  \qquand  a_{ij} := \dd[l^j]t\Big|_{c_i} \,.
 \end{equation}
Moreover, for spRK, we will also use the nodal weights and coefficients $(\bar b_j,\bar a_{ij})$ given by Equation \eqref{eq:spRK.coeff} and, for sG, the source and target coefficients
\[ \alpha^j := l^j(0)  \qquand  \beta^j:=l^j(1) \,. \]

Finally, if $L$ denotes a Lagrangian from $\bbR^n\times\bbR^n$ to $\bbR$ coupled with an external force $F\colon(q,\dq)\in\bbR^n\times\bbR^n\mapsto(q,F(q,\dq))\in\bbR^n\times\bbR^n$, then we define
\[ P_i := \pp[L]{\dq}\Big|_i = \pp[L]{\dq}\Big|_{(Q_i,\dQ_i)}  \qquand  \dP_i := \pp[L]{q}\Big|_i + F_i = \pp[L]{q}\Big|_{(Q_i,\dQ_i)} + F(Q_i,\dQ_i) \,, \]
where $(Q_i,\dQ_i)$ are couples of micro-nodes and micro-velocities given by each method. Besides, $D_i$ will stand for the partial derivative with respect to the $i$-th component.

\subsubsection{Symplectic-Partitioned Runge-Kutta Methods}\label{sec:sprk}
Although the variational de\-ri\-va\-tion of spRK methods in the framework of Geometric Mechanics is an already known fact (see \cite{MaWe01} for an ``intrinsic'' derivation, as the current, or \cite{HaLuWa02} for a ``constrained'' one), both based on the original works of \cite{Su90,SaCa94}, we present it here again in order to ease the understanding of and the comparison with sG methods below.

Given a point $q_0\in\bbR^n$ and vectors $\{\dQ_i\}_{i=1,\dots,s}\subset\bbR^n$, we define the polynomial curves
\[ \caldQ(t) := \sum_{j=1}^sl^j(t/h)\dQ_j \qquand \calQ(t) := q_0+h\sum_{j=1}^s\int_0^{t/h}l^j(\tau)\diff\tau\dQ_j \,. \]
We have
\begin{equation} \label{eq:spRK.node.rel}
\dQ_i = \dQ(h\cdot c_i) \quand Q_i := \calQ(h\cdot c_i) = q_0 + h\sum_{j=1}^sa_{ij}\dQ_j \,.
\end{equation}
Note that the polynomial curve $\calQ$ is uniquely determined by $q_0$ and $\{\dQ_i\}_{i=1,\dots,s}$. In fact, it is the unique polynomial curve $\calQ$ of degree $s$ such that $\calQ(0)=q_0$ and $\caldQ(h\cdot c_i)=\dQ_i$. However, if we define the configuration point
\begin{equation} \label{eq:spRK.target.node}
q_1 := \calQ(h\cdot1) = q_0 + h\sum_{j=1}^sb_j\dQ_j
\end{equation}
and consider it fixed, then $\calQ$ is uniquely determined by $q_0$, $q_1$ and the $\dQ_i$'s but one. Namely, take any $1\leq i_0 \leq s$ such that $b_{i_0}\neq0$ and fix it, then
\[ \dQ_{i_0} = \left(\frac{q_1-q_0}h-\sum_{j\neq i_0}b_j\dQ_j\right)/b_{i_0} \,. \]

We now define the multi-vector discrete Lagrangian
\begin{equation}\label{eq:multivecL}
L_d(\dQ_1,\dots,\dQ_s) := h\sum_{i=1}^sb_iL(Q_i,\dQ_i)
\end{equation}
and the multi-vector discrete force
\[  F_d(\dQ_1,\dots,\dQ_s)\cdot(\delta Q_1,\dots,\delta Q_s) := h\sum_{i=1}^sb_iF(Q_i,\dQ_i)\cdot\delta Q_i \,. \]
Although not explicitly stated, they both depend also on $q_0$. If we write the micro-node variations $\delta Q_i$ in terms of the micro-velocity variations $\delta\dQ_i$ (by definition \eqref{eq:spRK.node.rel}), we have that the multi-vector discrete force is
\[  F_d(\dQ_1,\dots,\dQ_s)\cdot(\delta Q_1,\dots,\delta Q_s) = h^2\sum_{j=1}^s\sum_{i=1}^sb_ia_{ij}F(Q_i,\dQ_i)\cdot\delta \dQ_j \,. \]
The two-point discrete Lagrangian is then
\begin{equation}\label{eq:twopointL}
 L_d(q_0,q_1) := \ext_{\calP^s([0,h],\bbR^n,q_0,q_1)}L_d(\dQ_1,\dots,\dQ_s)
 \end{equation}
where $\calP^s([0,h],\bbR^n,q_0,q_1)$ is the space of polynomials $\calQ$ of order $s$ from $[0,h]$  to $\bbR^n$ such that $\calQ(0)=q_0$ and $\calQ(h)=q_1$ and the vectors $\dQ_i$'s determine such polynomials as discussed above. The so called ``extremal'' is realized by a polynomial $\calQ\in\calP^s([0,h],\bbR^n,q_0,q_1)$ such that
\begin{equation} \label{eq:DEL.micro.spRK}
\delta L_d(\dQ_1,\dots,\dQ_s)\cdot(\delta \dQ_1,\dots,\delta \dQ_s) + F_d(\dQ_1,\dots,\dQ_s)\cdot(\delta Q_1,\dots,\delta Q_s) = 0
\end{equation}
for any variations $(\delta \dQ_1,\dots,\delta \dQ_s)$, taking into account that $\delta q_0=\delta q_1=0$ and that $\delta \dQ_{i_0} = \sum_{j\neq i_0}\pp*[\dQ_{i_0}]{\dQ_j}\delta \dQ_j$. For convenience, the previous equation \eqref{eq:DEL.micro.spRK} is developed afterwards.

The two-point discrete forward and backward forces are then
\begin{equation} \label{eq:twopoint_dF}
F_d^\pm(q_0,q_1)\cdot\delta(q_0,q_1) := h\sum_{i=1}^sb_iF(Q_i,\dQ_i)\cdot\pp[Q_i]{q_\pm}\delta q_\pm \,,
\end{equation}
where $q_-=q_0$ and $q_+=q_1$. Using the previous relations, we may write
\[ F_d^- = h\sum_{i=1}^sb_i(1-a_{ii_0}/b_{i_0})F_i \qquand F_d^+ = h\sum_{i=1}^sb_ia_{ii_0}/b_{i_0}F_i \,. \]

By the momenta-matching rule \eqref{eq:discrete.legendre.transform}, we have that
\begin{eqnarray*}
-p_0 &=& -D_{i_0}L_d(\dQ_1,\dots,\dQ_s)/(hb_{i_0}) + D_{q_0}L_d(\dQ_1,\dots,\dQ_s) + F_d^- \,,\\
 p_1 &=& D_{i_0}L_d(\dQ_1,\dots,\dQ_s)/(hb_{i_0})  + F_d^+\,.
\end{eqnarray*}
where $D_{q_0}$ stands for the partial derivative with respect to $q_0$. Combining both equations, we obtain that
\[ D_{i_0}L_d + h^2\sum_{i=1}^sb_ia_{ii_0}F_i = hb_{i_0}p_1 \qquand  p_1 = p_0 + D_{q_0}L_d + h\sum_{i=1}^sb_iF_i \,. \]
Coming back to Equation \eqref{eq:DEL.micro.spRK}, we have that
\begin{align*}
0 &= \delta L_d(\dQ_1,\dots,\dQ_s)\cdot(\delta \dQ_1,\dots,\delta \dQ_s) + F_d(\dQ_1,\dots,\dQ_s)\cdot(\delta Q_1,\dots,\delta Q_s)\\
&= \sum_{j\neq i_0}\left( D_jL_d + h^2\sum_{i=1}^sb_ia_{ij}F_i + \pp[\dQ_{i_0}]{\dQ_j}\left( D_{i_0}L_d + h^2\sum_{i=1}^sb_ia_{ii_0}F_i \right) \right)\delta \dQ_j \,.
\end{align*}
Therefore, for $j\neq i_0$, we have that
\[ D_jL_d + h^2\sum_{i=1}^sb_ia_{ij}F_i = b_j/b_{i_0}\cdot\left( D_{i_0}L_d + h^2\sum_{i=1}^sb_ia_{ii_0}F_i \right) \,. \]

Thus, the integrator is defined by
\begin{subequations}
\begin{gather}
D_jL_d(\dQ_1,\dots,\dQ_s) + h^2\sum_{i=1}^sb_ia_{ij}F_i = hb_jp_1 \,, \\
q_1 = q_0 + h\sum_{j=1}^sb_j\dQ_j \,, \\
p_1 = p_0 + D_{q_0}L_d(\dQ_1,\dots,\dQ_s) + h\sum_{i=1}^sb_iF_i \,.
\end{gather}
\end{subequations}
Besides, using the definition of the discrete Lagrangian, we have
\begin{eqnarray*}
D_jL_d(\dQ_1,\dots,\dQ_s) + h^2\sum_{i=1}^sb_ia_{ij}F_i
&=& h^2\sum_{i=1}^sb_ia_{ij}\dP_i + hb_jP_j \,, \\
D_{q_0}L_d(\dQ_1,\dots,\dQ_s) + h\sum_{i=1}^sb_iF_i
&=& h\sum_{i=1}^sb_i\dP_i \,.
\end{eqnarray*}
Therefore, we may write
\begin{gather*}
P_j = p_0 + h\sum_{i=1}^sb_i(1-a_{ij}/b_j)\dP_i = p_0 + h\sum_{i=1}^s\bar a_{ji}\dP_i \,, \\
p_1 = p_0 + h\sum_{i=1}^sb_i\dP_i = p_0 + h\sum_{i=1}^s\bar b_i\dP_i \,,
\end{gather*}
were $\bar a_{ij}$ and $\bar b_i$ are given by
\begin{equation} \label{eq:spRK.coeff}
b_i\bar a_{ij}+\bar b_ja_{ji} = b_i\bar b_j \,,\quad b_i = \bar b_i\,.
\end{equation}

In summary, the equations that define the spRK integrator (with forces), are together with \eqref{eq:spRK.coeff}
\begin{subequations} \label{eq:spRK}
\begin{align}
q_1 =& q_0 + h\sum_{j=1}^sb_j\dQ_j\,, & p_1 =& p_0 + h\sum_{j=1}^s\bar b_j\dP_j\,,\label{eq:spRK_1}\\
Q_i =& q_0 + h\sum_{j=1}^sa_{ij}\dQ_j\,, & P_i =& p_0 + h\sum_{j=1}^s\bar a_{ij}\dP_j\,,\label{eq:spRK_2}\\
P_i =& \pp[L]{\dq}(Q_i,\dQ_i)\,, & \dP_i =& \pp[L]{q}(Q_i,\dQ_i) + F(Q_i,\dQ_i) \,.\label{eq:spRK_3}
\end{align}
\end{subequations}

\subsubsection{Symplectic Galerkin Methods}\label{sec:sg}
Galerkin methods are a class of methods to transform a problem given by a continuous operator (such as a differential operator) to a discrete problem. As such, spRK methods fall into the scope of this technique and could be also classified as ``symplectic Galerkin'' methods. However, we want to stress on the difference between what is called spRK in the literature and what we refer here as sG. The wording should not be confused by the one used in \cite{MaWe01}.

Given points $\{Q_i\}_{i=1,\dots,s}\subset\bbR^n$, we define the polynomial curves
\[ \calQ(t) := \sum_{j=1}^sl^j(t/h)Q_j \qquand \caldQ(t) := \frac1h\sum_{j=1}^s\dot l^j(t/h)Q_j \,. \]
We have
\[ Q_i = \calQ(h\cdot c_i) \quand \dQ_i := \caldQ(h\cdot c_i) = \frac1h\sum_{j=1}^sa_{ij}Q_j \,. \]
Note that the polynomial curve $\calQ$ is uniquely determined by the points $\{Q_i\}_{i=1,\dots,s}$. In fact, it is the unique polynomial curve $\calQ$ of degree $s-1$ such that $\calQ(h\cdot c_i)=Q_i$. However, if we define the configuration points
\begin{equation} \label{sG.source-target.nodes}
q_0 := \calQ(h\cdot0) = \sum_{j=1}^s\alpha^jQ_j  \qquand  q_1 := \calQ(h\cdot1) = \sum_{j=1}^s\beta^jQ_j
\end{equation}
and consider them fixed, then $\calQ$ is uniquely determined by $q_0$, $q_1$ and the $Q_i$'s but a couple. For instance, we may consider $Q_1$ and $Q_s$ as functions of the others, since the relations \eqref{sG.source-target.nodes} define a system of linear equations where the coefficient matrix has determinant $\gamma:=\alpha^1\beta^s-\alpha^s\beta^1\neq0$ (if and only if $c_1\neq c_s$). More precisely,
\[
   \left(
   \begin{array}{c}
     Q_1\\
     Q_s
   \end{array}
   \right)
   =
   \frac1\gamma
   \left(
   \begin{array}{cc}
      \beta^s & -\alpha^s\\
     -\beta^1 &  \alpha^1
   \end{array}
   \right)
   \left(
   \begin{array}{c}
     q_0-\sum_{j=2}^{s-1}\alpha^jQ_j\\
     q_1-\sum_{j=2}^{s-1}\beta^jQ_j
   \end{array}
   \right) \,.
\]

We now define the multi-point discrete Lagrangian
\begin{equation}\label{eq:multipointL}
L_d(Q_1,\dots,Q_s) := h\sum_{i=1}^sb_iL(Q_i,\dQ_i)
\end{equation}
and the multi-vector discrete force
\[  F_d(Q_1,\dots,Q_s)\cdot(\delta Q_1,\dots,\delta Q_s) := h\sum_{i=1}^sb_iF(Q_i,\dQ_i)\cdot\delta Q_i \,. \]
The two-point discrete Lagrangian is then
\begin{equation} \label{eq:twopointL_sG}
L_d(q_0,q_1) := \ext_{\calP^s([0,h],\bbR^n,q_0,q_1)}L_d(Q_1,\dots,Q_s)
\end{equation}
where $\calP^s([0,h],\bbR^n,q_0,q_1)$ is the space of polynomials $\calQ$ of order $s$ from $[0,h]$ to $\bbR^n$ such that the points $Q_i$'s determine such polynomials as discussed above. The so called ``extremal'' is realized by a polynomial $\calQ\in\calP^s([0,h],\bbR^n,q_0,q_1)$ such that
\begin{equation} \label{eq:DEL.micro.sG}
\delta L_d(Q_1,\dots,Q_s)\cdot(\delta Q_1,\dots,\delta Q_s) + F_d(Q_1,\dots,Q_s)\cdot(\delta Q_1,\dots,\delta Q_s) = 0
\end{equation}
for any variations $(\delta Q_1,\dots,\delta Q_s)$, taking into account that $\delta q_0=\delta q_1=0$ and that $\delta Q_i = \sum_{j=2}^{s-1}\pp*[Q_i]{Q_j}\delta Q_j$, $i=1,s$.
For convenience, the previous equation \eqref{eq:DEL.micro.sG} is developed afterwards.

The two-point discrete forward and backward forces are then formally defined by Equation \eqref{eq:twopoint_dF}. Using the previous relations, we may write
\[ F_d^- = h(\beta_sb_1F_1-\beta_1b_sF_s)/\gamma \qquand F_d^+ = h(\alpha_1b_sF_s-\alpha_sb_1F_1)/\gamma \,. \]

By the momenta-matching rule \eqref{eq:discrete.legendre.transform}, we have that
\begin{align*}
-p_0 &= \phantom{-} \beta^s/\gamma \cdot (D_1L_d + hb_1F_1) - \beta^1/\gamma \cdot (D_sL_d + hb_sF_s) \ \ \text{and}\\
 p_1 &= -\alpha^s/\gamma \cdot (D_1L_d + hb_1F_1) + \alpha^1/\gamma \cdot (D_sL_d + hb_sF_s) \,.
\end{align*}
By a linear transformation of both equations, we obtain
\begin{align*}
  D_1L_d(Q_1,\dots,Q_s) + hb_1F_1 &= -\alpha^1p_0 + \beta^1p_1 \ \ \text{and}\\
  D_sL_d(Q_1,\dots,Q_s) + hb_sF_s &= -\alpha^sp_0 + \beta^sp_1 \,.
\end{align*}
Coming back to Equation \eqref{eq:DEL.micro.sG}, we have that
\begin{align*}
0 &= \left(\delta L_d(Q_1,\dots,Q_s) + F_d(Q_1,\dots,Q_s)\right)\cdot(\delta Q_1,\dots,\delta Q_s) \\
&= \sum_{j=2}^{s-1}\left[(D_1L_d+hb_1F_1)\pp[Q_1]{Q_j} + (D_jL_d+hb_jF_j) + (D_sL_d+hb_sF_s)\pp[Q_s]{Q_j}\right]\delta Q_j \,.
\end{align*}
Therefore, for $j=2,\dots,s-1$, we obtain
\begin{align*}
\gamma (D_jL_d+hb_jF_j)
=\,& (\alpha^j\beta^s\!-\!\alpha^s\beta^j)(D_1L_d+hb_1F_1) + (\alpha^1\beta^j\!-\!\alpha^j\beta^1)(D_sL_d+hb_sF_s)\\
=\,& (\alpha^1\beta^s\!-\!\alpha^s\beta^1)(\beta^jp_1\!-\!\alpha^jp_0) \,.
\end{align*}

Thus, the integrator is defined by
\begin{subequations}
\begin{gather}
\label{eq:sG.DEL}
D_jL_d(Q_1,\dots,Q_s) + hb_jF_j = -\alpha^jp_0 + \beta^jp_1 \,,\ j=1,\dots,s;\\
q_0 = \sum_{j=1}^s\alpha^jQ_j  \qquand q_1 = \sum_{j=1}^s\beta^jQ_j
\end{gather}
\end{subequations}
Besides, using the definition of the discrete Lagrangian, we have
\begin{eqnarray*}
D_jL_d(Q_1,\dots,Q_s)
&=& h\sum_{i=1}^sb_i\left( \pp[L]{q}\Big|_i\pp[Q_i]{\dQ_j} + \pp[L]{\dq}\Big|_i\pp[\dQ_i]{\dQ_j} \right) \\
D_jL_d(Q_1,\dots,Q_s) + hb_jF_j
&=& hb_j\dP_j + \sum_{j=1}^sb_ia_{ij}P_i \,.
\end{eqnarray*}
Therefore, we may simply write
\[ hb_j\dP_j + \sum_{j=1}^sb_ia_{ij}P_i = -\alpha^jp_0 + \beta^jp_1 \,. \]

In summary and for a proper comparison, we write the equations that define the sG integrator (with forces) in a pRK way, that is
\begin{subequations} \label{eq:sG}
\begin{align}
  q_0 =& \sum_{j=1}^s\alpha^jQ_j\,, & q_1 =& \sum_{j=1}^s\beta^jQ_j\,,\label{eq:sG_1}\\
\dQ_i =& \frac1h\sum_{j=1}^sa_{ij}Q_j\,, & \dP_i =& \frac{\beta^ip_1-\alpha^ip_0}{h\bar b_i} + \frac1h\sum_{j=1}^s\bar a_{ij}P_j\,,\label{eq:sG_2}\\
  P_i =& \pp[L]{\dq}(Q_i,\dQ_i)\,, & \dP_i =& \pp[L]{q}(Q_i,\dQ_i) + F(Q_i,\dQ_i)\,,\label{eq:sG_3}
\end{align}
\end{subequations}
where $b_ia_{ij}+\bar b_j\bar a_{ji}=0$ and $b_i=\bar b_j$.

We remark that Equation \eqref{eq:sG.DEL} generalizes the ones obtained in \cite{CaJuOb12,Le04}, where the collocation points are chosen such that $c_1=0$ and $c_s=1$, which is a rather particular case.

\subsubsection{Similarities and differences between spRK and sG}\label{sec:rel}
As already mentioned, both methods can be considered of Galerkin type. In this sense, spRK and sG could be refereed as a symplectic Galerkin integrators of 1st and 0th kind, respectively, since spRK is derived from the 1st derivative of an extremal polynomial and sG from the polynomial itself. At this point, a very natural question could arise: Are spRK and sG actually two different integrator schemes? Even though the derivations of both methods are quite similar, they are in general different (although they could coincide for particular choices of the Lagrangian, the collocation points and the integral quadrature). A weak but still fair argument to support this is that, at each step, spRK relies on the determination of the micro-velocities $\dQ_i$, while sG does so on the micro-nodes $Q_i$. All the other ``unknowns'' are then computed from the determined micro-data.

In the simplest of the cases, that is, the case where one considers a Lagrangian of the form kinetic minus potential energy, $L(q,\dq) = \frac12\dq^TM\dq - U(q)$, with $M$ a constant mass matrix; $s=2$ micro-nodes (inner-stages); and Lobatto's quadrature, $c_1=0$, $c_2=1$; one may show that both schemes, spRK \eqref{eq:spRK} and sG \eqref{eq:sG}, reduce to the well-known \emph{leap-frog \emph{or} Verlet method}. They will differ when the previous main assumptions are violated, for instance if $M$ is not constant or the quadrature is other than Lobatto's.

\begin{example}
We consider a Lagrangian with a scalar mass matrix dependent on the configuration, that is, a Lagrangian of the form $L(q,\dq) = \frac12\lambda(q)\|\dq\|^2 - V(q)$, with $\lambda\colon Q\to\bbR$. Under this assumption and noting $\lambda_{1/2} := \frac{\lambda_0+\lambda_1}2$, $(\nabla)\lambda_i := (\nabla)\lambda(q_i)$, $(\nabla)V_i := (\nabla)V(q_i)$, $i=0,1$, the spRK scheme \eqref{eq:spRK} as well as the  the sG scheme \eqref{eq:sG} reduce to
\begin{align*}
p_{1/2} =\,& p_0 + \frac{h}2\left(\boxed{\frac{\nabla\lambda_0}{2\lambda_{\mathsf a}^2}}\|p_{1/2}\|^2-\nabla V_0\right) \,,\\
q_1 =\,& q_0 + \frac{h}2\left(\boxed{\frac1{\lambda_{\mathsf a}}+\frac1{\lambda_{\mathsf b}}}\right)p_{1/2} \,,\\
p_1 =\,& p_{1/2} + \frac{h}2\left(\boxed{\frac{\nabla\lambda_1}{2\lambda_{\mathsf b}^2}}\|p_{1/2}\|^2-\nabla V_1\right) \,,
\end{align*}
with a slight difference in the subindexes appearing in the framed factors. While in the spRK scheme, ${\mathsf a}=0$ and ${\mathsf b}=1$; in the sG scheme ${\mathsf a}={\mathsf b}=1/2$. It is important to note that, even though the difference is small, it makes both schemes certainly different. Besides notice that the first two equations define $p_{1/2}$ and $q_1$ implicitly and that the whole set reduces to the Verlet method for a constant $\lambda$. Indeed, it is shown in \cite{Ca13,O14} that for a Lagrangian with constant mass matrix and Lobatto quadrature rule, the sG and the spRK method coincide.
\end{example}

\subsubsection{Order of the schemes}
With respect to the accuracy of the schemes, for any Gaussian--like quadrature (Gauss-Legendre, Gauss-Lobatto, Radau and Chebyshev) and any method (spRK and sG), the schemes have a convergence order up to $2s$ (which is only attained by the combination Gauss-Lobatto together with spRK) but no lower than $2s-2$, being $s$ the number of internal stages, see Table \ref{tab:comparison}. We emphasize that these orders have been determined numerically experimenting with several ``toy examples'' for which exact solutions are known, \eg\ the harmonic oscillator and the 2 body problem (see Section \ref{sec:conmutation}), however they coincide with the known analytic results when available, that is spRK together with the Gauss-Legendre or Gauss-Lobatto quadratures (see \eg\ \cite{HaLuWa02}).

\begin{table}[h]
  \centering
  \begin{tabular}{cc|c|c|}
    \cline{3-4}
    & & spRK & sG \\
    \cline{1-4}
    \multicolumn{2}{|c|}{micro-data} & $\dQ_i$ & $Q_i$ \\
    \cline{1-4}
    \multicolumn{2}{|c|}{polynomial degree} & $s$ & $s-1$ \\
    \cline{1-4}
    \multicolumn{2}{|c|}{variational eq.'s} & $s+1$ & $s$ \\
    \cline{1-4}
    \multicolumn{2}{|c|}{extra equations} & $1$ & $2$ \\
    \cline{1-4}
    \multicolumn{1}{|c|}{\multirow{4}{*}{quadrature}} & Gauss-Legendre & $2s$ & $2s-2$ \\
    \cline{2-4}
    \multicolumn{1}{|c|}{\multirow{4}{*}{}} & Gauss-Lobatto & $2s-2$ & $2s-2$ \\
    \cline{2-4}
    \multicolumn{1}{|c|}{\multirow{4}{*}{}} & Radau & $2s-1$ & $2s-2$ \\
    \cline{2-4}
    \multicolumn{1}{|c|}{\multirow{4}{*}{}} & Chebyshev & $2s-2$ & $2s-2$ \\
    \cline{1-4}
    & & \multicolumn{2}{|c|}{order method} \\
    \cline{3-4}
  \end{tabular}

  \caption{Comparison of $s$-stage variational integrators.}
  \label{tab:comparison}
\end{table}

\subsection{Discrete optimal control problem}\label{subsec:docp}
For the discretization of the optimal control problem~\ref{prob:locp}, we employ the class of high order variational integrators. By choosing an appropriate approximation $J_d$ of the cost functional $J$, the general discrete Lagrangian optimal control problem as discretization of the Langrangian optimal control problem~\ref{prob:locp} reads (see also \cite{ObJuMa10})
\begin{problem}[Discrete Lagrangian optimal control problem]\label{probd:locp}
\begin{subequations}\label{eq:gDLOCP}
\begin{equation}\label{eq:gDLOCP1}
\min_{\{q_k,u_k\}_{k=0}^N} J_d(\{q_k,u_k\}_{k=0}^N)
\end{equation}
subject to
\begin{align}
q_0 &= q^0,\label{eq:gDLOCP2}\\
D_2L(q^0,\dq^0)+D_1L_d(q_0,q_1) + F_0^-& = 0,\label{eq:gDLOCP3}\\
D_2 L_d (q_{k-1},q_k) + D_1 L_d (q_k,q_{k+1}) + F_d^+(q_{k-1},q_k) + F_d^-(q_k,q_{k+1}) &= 0,\label{eq:gDLOCP4}
\end{align}
\end{subequations}
for $k=1,\dots, N-1$ and where Equation~\eqref{eq:gDLOCP4} is the forced discrete Euler-Lagrange equation defined in \eqref{eq:DEL} and Equations~\eqref{eq:gDLOCP2}-\eqref{eq:gDLOCP3} correspond to the initial condition~\eqref{eq:LOCP_ini} expressed by means of the discrete Legendre transform \eqref{eq:discrete.legendre.transform}. Here, the control trajectory $u$ is approximated by the discrete values $u_k$, $k=0,\ldots,N$, such that $u_k \approx u(t_k)$. Note that for the controlled case, the $F_d^\pm$ are dependent on $u_k$. To specify the discrete optimal control problem, in particular, the approximation of the control parameter $u$ and the cost functional $J$, we focus on the high order variational integrators derived in Section~\ref{sec:hovi}, namely the spRK and the sG method, and find a discrete version of the more general optimal control problem~\ref{prob:ocp}.
\end{problem}

As for the approximation of $q(t)$ and $\dq(t)$, we also use a polynomial for the approximation of the control function $u(t)$ on $[0,h]$. For a given set of collocation points $0\leq c_1<\ldots<c_s\leq1$ and given control points $\{ U_i \}_{i=1,\ldots,s}\subset U$ we define the polynomial of degree $s-1$
\[ \calU(t):=\sum_{j=1}^s l^j(t/h)U_j \]
such that $U_i = \calU(h\cdot c_i),\, i=1,\ldots,s$. Note that the control polynomial $\calU(t)$ has the same degree as the polynomial $\calQ(t)$ for the sG integrator, whereas for the spRK integrator it coincides with the polynomial degree of $\caldQ(t)$. To take in consideration the control dependent force in the previous derivation of the spRK and the sG schemes into account, we replace in the definitions for $F_d^\pm$ in Equation \eqref{eq:twopoint_dF} the external force $F_i=F(Q_i,\dQ_i)$ by $F_i=F(Q_i,\dQ_i,U_i)$. Furthermore, for a regular Lagrangian and by using the definitions for $f$ and $g$ in \eqref{eq:fg}, we can write Equations~\eqref{eq:spRK_3} and \eqref{eq:sG_3} as
\[f(Q_i, P_i) = \left(\pp[L]{\dq}\right)^{-1}(Q_i,P_i)\,, \quad g(Q_i, P_i,U_i) = \pp[L]{q}(Q_i,\dQ_i) + F(Q_i,\dQ_i,U_i) \]
such that the spRK scheme can be written as
\begin{subequations} \label{eq:spRKred}
\begin{align}
q_1 =& q_0 + h\sum_{j=1}^sb_j f(Q_j,P_j)\,, & p_1 =& p_0 + h\sum_{j=1}^s\bar b_jg(Q_j,P_j,U_j)\,,\label{eq:spRK_red1}\\
Q_i =& q_0 + h\sum_{j=1}^sa_{ij} f(Q_j,P_j)\,, & P_i =& p_0 + h\sum_{j=1}^s\bar a_{ij}g(Q_j,P_j,U_j)\,,\label{eq:spRK_red2}
\end{align}
\end{subequations}
with $b_i\bar a_{ij}+\bar b_ja_{ji} = b_i\bar b_j$ and $b_i=\bar b_j$
and the sG scheme reduces to
\begin{subequations} \label{eq:sG_general}
\begin{align}
  q_0 =& \sum_{j=1}^s\alpha^jQ_j\,, & q_1 =& \sum_{j=1}^s\beta^jQ_j\,,\\
f(Q_i,P_i) =& \frac{1}{h}\sum_{j=1}^sa_{ij}Q_j\,, & g(Q_i,P_i,U_i) =& \frac{\beta^ip_1-\alpha^ip_0}{h\bar b_i} + \frac1h\sum_{j=1}^s\bar a_{ij}P_j\,,
\end{align}
\end{subequations}
where $b_ia_{ij}+\bar b_j\bar a_{ji}=0$ and $b_i=\bar b_j$.
Remember that the coefficients $a_{ij}$ are different for the two schemes \eqref{eq:spRKred} and \eqref{eq:sG_general} (see \eqref{eq:coeff_ab}).
To approximate the cost functional $\int_0^h C(q(t),p(t),u(t))\,\dt$ in \eqref{eq:OCP:cost} we employ the same quadrature rule that we use to approximate the action on $[0,h]$ (\cf\ \eqref{eq:multivecL} and \eqref{eq:multipointL}) such that the discrete density cost function $C_d$ is defined by
\[
C_d(\{Q_i,P_i,U_i\}_{i=1}^s) := h \sum_{i=1}^s b_i C(Q_i,P_i,U_i) \approx \int_0^h C(q(t),p(t),u(t))\,\dt.
\]

So as to prevent a proliferation of symbols and alleviate the notation, along a time step interval $[t_k,t_{k+1}]$, we write $q_h^k$, $p_h^k$ and $u_h^k$ instead of $\{q_k,\{Q_i^k\}_{i=1}^s,q_{k+1}\}$, $\{p_k,\{P_i^k\}_{i=1}^s,p_{k+1}\}$ and $\{U_i^k\}_{i=1}^s$, respectively. We drop the superscript $k$ if we consider an arbitrary time step interval $[0,h]$. With some abuse, along the whole interval $[0,T]$, we equally write $q_h$, $p_h$ and $u_h$ instead of $\{\{q_k,Q_i^k\}_{i=1,\dots,s}^{k=0,\dots,N-1},q_N\}$, $\{\{p_k,P_i^k\}_{i=1,\dots,s}^{k=0,\dots,N-1},p_N\}$ and $\{U_i^k\}_{i=1,\dots,s}^{k=0,\dots,N-1}$, respectively.

With this notation we define the discrete cost function $J_d$ as
\[
J_d(q_h, p_h,u_h) := \sum_{k=0}^{N-1} C_d(q_h^k,p_h^k,u_h^k) + \Phi(q_N,p_N)
\]
and introduce the following two discrete optimal control problems, where the discretization in Problem~\ref{prob:docp:spRK} is based on the spRK integrator and in Problem~\ref{prob:docp:sG} on the sG integrator.

\begin{problem}[Discrete optimal control problem: the spRK case] \label{prob:docp:spRK}
\begin{subequations}\label{eq:docp:spRK}
\begin{equation}\label{eq:docp:spRK:a}
\min_{q_h, p_h,u_h} J_d(q_h, p_h, u_h)
\end{equation}
subject to
\begin{align}
\label{eq:docp:spRK:b}
q_{k+1} =& q_k + h\sum_{j=1}^sb_j f(Q^k_j,P^k_j)\,, & p_{k+1} =& p_k + h\sum_{j=1}^s\bar b_jg(Q^k_j,P^k_j,U^k_j)\,,\\
\label{eq:docp:spRK:c}
Q^k_i =& q_k + h\sum_{j=1}^sa_{ij} f(Q^k_j,P^k_j)\,, & P^k_i =& p_k + h\sum_{j=1}^s\bar a_{ij}g(Q^k_j,P^k_j,U^k_j)\,
\end{align}
$k=0,\ldots,N-1,\, i=1,\ldots,s$, with $b_i\bar a_{ij}+\bar b_ja_{ji} = b_i\bar b_j$ and $b_i=\bar b_j$,
\begin{eqnarray}
\label{eq:docp:spRK:d}
(q_0,p_0) &=& (q^0,p^0), \quad U_i^k \in U.
\end{eqnarray}
\end{subequations}
\end{problem}

\begin{problem}[Discrete optimal control problem: the sG case]\label{prob:docp:sG}
\begin{subequations}\label{eq:docp:sG}
\begin{equation}\label{eq:docp:sG:a}
\min_{q_h, p_h, u_h} J_d(q_h, p_h, u_h)
\end{equation}
subject to
\begin{align}
\label{eq:docp:sG:b}
  q_k =& \sum_{j=1}^s\alpha^jQ^k_j\,, & q_{k+1} =& \sum_{j=1}^s\beta^jQ^k_j\,,\\
\label{eq:docp:sG:c}
f(Q^k_i,P^k_i) =& \frac{1}{h}\sum_{j=1}^sa_{ij}Q^k_j\,, & g(Q^k_i,P^k_i,U^k_i) =& \frac{\beta^ip_{k+1}-\alpha^ip_k}{h\bar b_i} + \frac1h\sum_{j=1}^s\bar a_{ij}P^k_j\,,
\end{align}
$k=0,\ldots,N-1,\, i=1,\ldots,s$, with $b_ia_{ij}+\bar b_j\bar a_{ji}=0$ and $b_i=\bar b_j$,
\begin{eqnarray}
\label{eq:docp:sG:d}
(q_0,p_0) &=& (q^0,p^0), \quad U_i^k \in U.
\end{eqnarray}
\end{subequations}
\end{problem}

Since Problem~\ref{prob:docp:spRK} has been extensively studied in \cite{ObJuMa10} (as discussed in Section~\ref{subsec:comp}), in this work we focus on Problem~\ref{prob:docp:sG}.

\subsection{Comparison of different solution methods}\label{subsec:comp}

In Figure~\ref{fig:intro_comparison} (see also \cite{ObJuMa10}) we present schematically different discretization strategies for optimal control problems.
Starting with the Lagrangian optimal control problem~\ref{prob:locp}, we obtain via variation (for the derivation of the Euler-Lagrange equations) the optimal control problem~\ref{prob:ocp}. For its solution, direct or indirect methods can be employed (the differences of direct and indirect methods are already discussed in Section~\ref{sec:intro}).

In the DMOC approach, rather than discretizing the differential equations arising from the Lagrange-d'Alembert principle, we discretize in the earliest stage, namely already at the level of the variational principle. Then, we perform the variation only on the discrete level which results in a nonlinear programming problem (in particular we obtain the discrete Lagrangian optimal control problem~\ref{probd:locp}). Its necessary optimality conditions are derived by a dualization step as for a standard direct method.
This approach that uses the concept of discrete mechanics leads to a special discretization of the system equations based on variational integrators. Thus, the discrete optimal control problem inherits special properties exhibited by variational integrators as extensively discussed in \cite{ObJuMa10}.

\begin{figure}[h]
\begin{center}
\includegraphics[width=\textwidth]{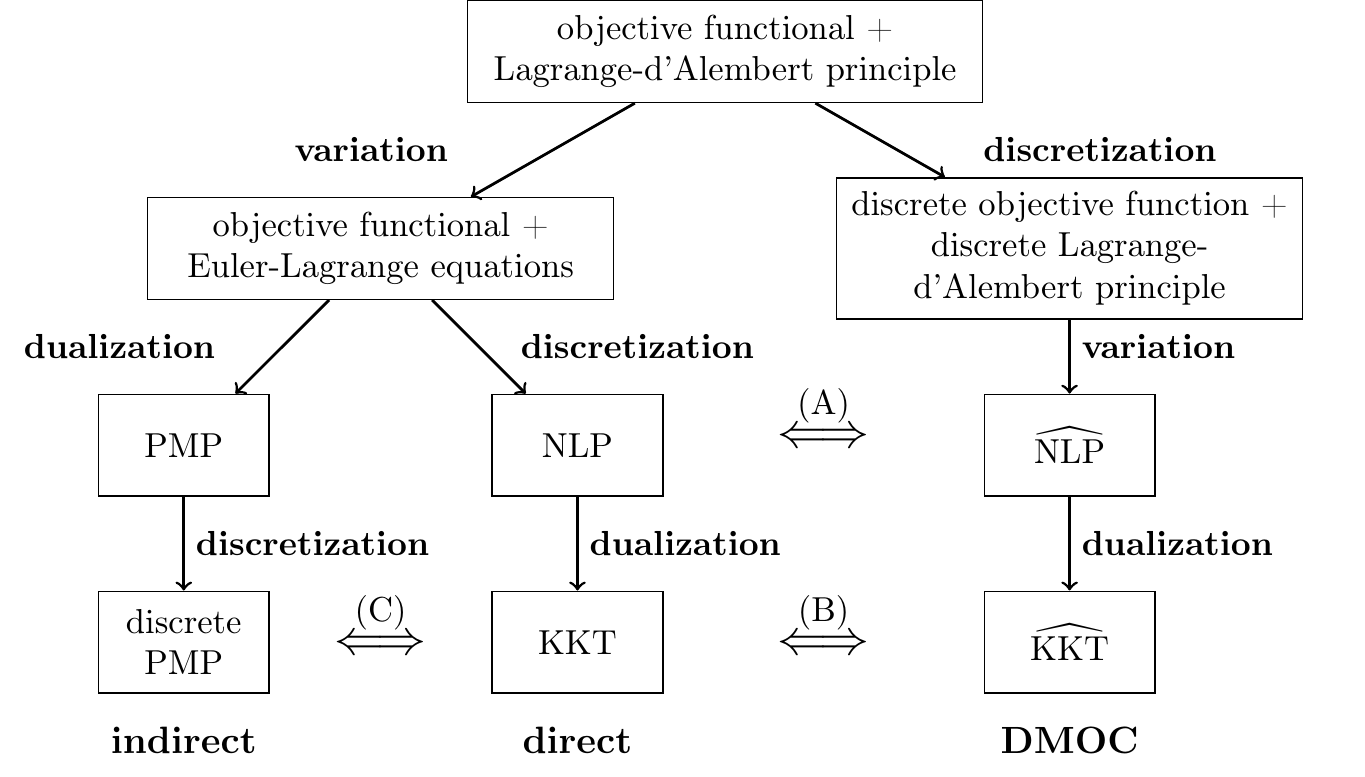}
\caption[Diagram of OCP numerical methods]{\footnotesize Optimal control for mechanical systems: the order of variation, dualization and discretization for deriving the necessary optimality conditions.}\label{fig:intro_comparison}
\end{center}
\end{figure}

In this work we are interested in the question under which conditions the discretized necessary optimality conditions (discrete PMP) and the KKT conditions resulting from the discrete optimal control problems \eqref{prob:docp:spRK} and \eqref{prob:docp:sG} ($\widehat{\text{KKT}}$) are identical.
To this end, we summarize by now already known equivalence relations (A), (B) and (C) as indicated in Figure~\ref{fig:intro_comparison}.

\paragraph{Equivalence (A)}
This equivalence corresponds to the commutation of variation and discretization.
For particular variational integrators their equivalence to other well-known integration methods has been shown, \eg\ the equivalence to the St\"ormer-Verlet method (\cite{Ma92}), to the Newmark algorithm (\cite{Kane00}) or, more generally, to spRK methods (\cite{Su90, MaWe01, HaLuWa02}) applied to the corresponding Hamiltonian system. Whereas the variational derivation of spRK methods is an already known fact and was presented in Section~\ref{sec:sprk} from a slightly different point of view, we found in Section~\ref{sec:sg} a new class of integrators, the sG methods, that applied to the Hamiltonian systems are equivalent to variational integrators based on the discrete Lagrangian~\eqref{eq:twopointL_sG}.
Obviously, if the variation of the discrete Lagrange-d'Alembert principle (the variational integrator) and the discretization of the Euler-Lagrange or, equivalently, the Hamiltonian equations is the same, then, using the same discretization for the cost functional provides the same NLP in the middle and the right branch in Figure~\ref{fig:intro_comparison}.

\paragraph{Equivalence (B)}
Of course, if NLP and $\widehat{\text{NLP}}$ are identical (equivalence (A)), then also the KKT and the $\widehat{\text{KKT}}$ conditions are identical.

\paragraph{Equivalence (C)} This equivalence corresponds to the commutation of discretization and dualization.
As already summarized in Section~\ref{sec:intro}, for the class of Legendre pseudospectral methods the commutation property has been proved if additional closure conditions are satisfied (see \cite{GoRoKaFa08,RoFa03}). In \cite{ObJuMa10} the commutation of discretization and dualization is proved for spRK methods, that means, the dualization of Problem~\ref{prob:docp:spRK} is the same as the discretization of the necessary optimality conditions using a spRK method. Due to equivalence (A), this commutation property also holds for high order variational integrators based on the discrete Lagrangian~\eqref{eq:twopointL}.
However, for general classes of variational integrators (in particular there exist high order variational integrators that are not equivalent to RK methods, see \eg\ \cite{O14}), the commutation property is still an open question.
As new contribution in this paper, we show in Section~\ref{sec:conmutation} that discretization and dualization also commute for the sG integrator (see Theorem~\ref{th:commutation}) (\ie\ the dualization of Problem~\ref{prob:docp:sG} is the same as the discretization of the necessary optimality conditions using an sG method).
We therefore find (besides the spRK methods) another class of variational integrators that fulfills the commutation property.

\section{Analysis of convergence of the primal variables} \label{sec:conv}
In this section, we present one of the main results of the paper, that is the convergence of the primal variables. Before that, a couple of comments are necessary to clarify the assumptions and notation that appear in the statement. Examples will enlighten the result.

We say that a function $f\colon H\to\bbR$ is \emph{coercive}, where $H$ is a Hilbert space with norm $\|\cdot\|$ (in our case of study $H$ is either $\bbR^m$ or $L^2([0,T],\bbR^m)$), if there exists a scalar factor $\alpha>0$ for which
\[ \liminf_{\|x\|\to\infty}\frac{f(x)}{\|x\|^2}\geq\alpha \,. \]
If $f$ depends on a further variable, $f=f(x,y)$, we say that $f$ is \emph{uniformly coercive in $x$ (with respect to $y$)} if the coercivity factor $\alpha$ does not depend on $y$.

In the next result, a discrete trajectory $q_h$, either over a time step interval $[0,h]$ or over the whole time interval $[0,T]$, should be understood as a continuous trajectory along $[0,h]$ or $[0,T]$, respectively. To that, on  $[0,h]$, $q_h=\{q_0,\{Q_i\}_{i=1}^s,q_1\}$ can be also viewed as its own linear interpolation, that is as the piecewise-linear continuous function $q_h\colon[0,h]\to\bbR^n$ determined by
\[ q_h(0)=q_0 \,,\ q_h(i\cdot h/(s+1))=Q_i\,,\ i=1,\ldots,s\,,\ q_h(h)=q_1 \]
and linear in between. One proceeds similarly on $[0,T]$ and as well for $p_h$ and $u_h$.

\begin{theorem} \label{thm:control:conv}
Given a Lagrangian function $L\colon TQ\to\bbR$, an external control force $F\colon TQ\times U\to T^*Q$, a density cost function $C\colon T^*Q\times U\to\bbR$ and a set of collocation points $0\leq c_1<\ldots<c_s\leq1$ defining a quadrature rule $(c_i,b_i)$, let us assume that
\begin{enumerate}
\renewcommand{\theenumi}{H\arabic{enumi}}
\item $L$ is regular; \label{thm:control:conv:regularity}
\item $F$ is affine on the controls, \ie\ $F(q,\dq,u)=F_0(q,\dq)+u\cdot F_1(q,\dq)$; \label{thm:control:conv:linearity}
\item $C$ is uniformly coercive in $u$ and smooth in $(q,p)$; \label{thm:control:conv:coercivity}
\item $(OCP)$, the continuous Problem \ref{prob:ocp}, has a unique solution $(\bar q,\bar p, \bar u)$;\label{thm:control:conv:uniqueness}
\item $b_i>0$ for $i=1,\ldots,s$; and \label{thm:control:conv:positiveness}
\item the associated spRK or sG scheme is convergent (for $L$, $F$ and any fixed $u$).\label{thm:control:conv:convergence}
\end{enumerate}
Then $(\bar q_h,\bar p_h, \bar u_h)$ converges (up to subsequence) to $(\bar q,\bar p, \bar u)$ as $h\to0$ ($N\to\infty$), strongly in $(q,p)$ and weakly in $u$, where $(\bar q_h,\bar p_h, \bar u_h)$ is the solution to $(OCP)_h$, the corresponding discrete Problem \ref{prob:docp:spRK} or \ref{prob:docp:sG}.
\end{theorem}

\begin{proof}
The assumption of coercivity \ref{thm:control:conv:coercivity} on the density cost function $C$ (recall $\Phi$ is assumed to be bounded from below) implies the coercivity of the total cost functional $J$ and, together with the assumption of positiveness \ref{thm:control:conv:positiveness} on the weight coefficients $b_i$, ensures also the coercivity of the discrete density cost $C_d$ and of the discrete total cost $J_d$: Indeed, from the definition and for $u_h$ big enough,
\[\frac{C_d(q_h,p_h,u_h;h)}{\|u_h\|^2} \geq h\min_ib_i\alpha \qquand \frac{J_d(q_h,p_h,u_h;h)}{\|u_h\|^2} \geq T\min_ib_i\alpha \,, \]
where $(q_h,p_h,u_h)$ represents a one step trajectory on the left while a full trajectory  on the right. The coercivity of $J$ follows similarly by integration. We stress the fact that $J_d$ is uniformly coercive with respect to the time step $h$.

Besides we have from \ref{thm:control:conv:regularity} that, for $h$ small enough, the discrete Lagrangian $L_d$ is regular. Therefore for each control set $u_h$, a unique solution $(q_h,p_h)$ to the discrete Euler-Lagrange equations exists. The minimization in $(OCP)_h$ is then nothing but a finite dimensional minimization problem whose solution existence is provided by the coercivity of $C_d$. We denote $(\bar q_h,\bar p_h, \bar u_h)$ such solution.

Since the discrete cost $J_d$ is uniformly coercive with respect to $h$, the sequence $(\bar u_h)$, thought in $L^2([0,T],U)$, is bounded. Hence, up to subsequence, it converges to some control $\tilde u\in L^2([0,T],U)$, for the weak topology of $L^2$. It only remains to show that $(\tilde q,\tilde p, \tilde u)$, where $(\tilde q,\tilde p)$ is the unique solution of the mechanical system corresponding to the control $\tilde u$ (hypothesis \ref{thm:control:conv:regularity}), is in fact the same point as $(\bar q,\bar p, \bar u)$.

Firstly, by \ref{thm:control:conv:linearity}, the controls enter the continuous and discrete dynamical equations \eqref{eq:LOCP_dA}, \eqref{eq:spRK_1}, \eqref{eq:spRK_2}, \eqref{eq:sG_1} and \eqref{eq:sG_2} linearly. It follows that $(\bar q_h,\bar p_h)$ converges uniformly to $(\tilde q,\tilde p)$ (more technical details on this standard reasoning may be found in \cite{Tr00}). Secondly, we observe that the control $\tilde u$ is optimal: Indeed, since the discrete cost $J_d$ is an approximation to the continuous one $J$, we have that for some exponent $r\geq1$
\begin{eqnarray*}
J(\bar q_h,\bar p_h,\bar u_h)
&  = & J_d(\bar q_h,\bar p_h,\bar u_h;h) + O(h^r)\\
&\leq& J_d((\bar q,\bar p,\bar u)_h;h) + O(h^r)\\
&  = & J(\bar q,\bar p,\bar u) + O(h^r) \,,
\end{eqnarray*}
where $(\bar q,\bar p,\bar u)_h$, for each $h$, represents simply the evaluation of $(\bar q,\bar p,\bar u)$ at the collocation points of each time interval. Passing to the limit, from the lower semi-continuity of $J$ with respect to $u$ (given by the integral expression and hypothesis \ref{thm:control:conv:coercivity}), it follows that
\[ J(\tilde q,\tilde p,\tilde u) \leq J(\bar q,\bar p,\bar u) \,. \]
Hence $(\tilde q,\tilde p,\tilde u)$ is optimal and by uniqueness coincides with $(\bar q,\bar p,\bar u)$.
\end{proof}

\begin{remark}\label{rem:coercivity}
Even though the rather ``classical'' definition of coercivity used here is perhaps less restrictive than the one given in \cite{Ha00,Ha01}, there is no direct relation between them in the sense that none of them implies or is implied by the other. They have non-empty intersection. In there, the coercivity is a classical sufficient second-order assumption ensuring the absence of conjugate points. In here, the coercivity implies the existence of a (global) optimal control. The framework that \ref{thm:control:conv} proposes has some advantages: It permits to prove without difficulty the existence of solutions for the discretized problem and it ensures the convergence from a simple topological argument. Moreover, the proof itself has the potential of being more general, for instance to consider final constraints (by means of finer arguments, the concept of end-point mapping, the general conjugate point theory).
\end{remark}

\begin{remark}\label{rem:nonlinearity}
It is worth to note that the previous proof withstands some easy generalizations. If we now take more general dynamics (nonlinear force in $u$) and costs, then the above reasoning works as well provided the cost is coercive in some $L^p$ and the dynamics satisfy, for instance,
\[ \limsup_{\|u\|\to\infty}\frac{\|\Psi(q,p,u)\|}{\|u\|^r} = 0 \,,\ p>r\,, \]
where $\Psi=0$ stands for the dynamical constraints \eqref{eq:LOCP_dA}.
\end{remark}

\begin{remark}\label{rem:convergence}
A convergence proof (including consistency and stability) for general variational integrators (as assumed in \ref{thm:control:conv:convergence}) is topic of ongoing research. For particular classes, the convergence is proven by showing that the variational integrator is equivalent to another well-known convergent method, as for example for symplectic partitioned Runge-Kutta methods. For a recent convergence analysis for Galerkin variational integrators by means of variational error analysis we refer to \cite{HaLe13}. The assumption \ref{thm:control:conv:uniqueness}, the uniqueness of the solution of $(OCP)$, is a classical one. It can be weakened by stating the result in terms of closure points as follows. We assume that $C^0([0,T],T^*Q)$ is endowed with the uniform convergence topology and that $L^2([0,T],U)$ is endowed with the weak topology. Then every closure point of the family of solutions $(\bar q_h,\bar p_h, \bar u_h)$ of $(OCP)_h$ in $C^0([0,T],T^*Q)\times L^2([0,T],U)$ is a solution of $(OCP)$.
\end{remark}

\begin{remark}\label{rem:discretization}
In the formulation of the previous discrete optimal control problems \ref{prob:docp:spRK} and \ref{prob:docp:sG}, we have chosen to discretize the control parameter and the cost functional in accordance to the dynamical discretization. Other possibilities are available which must be pondered. Let's assume temporarily that, besides the original set of collocation points $0\leq c_1<\ldots<c_s\leq1$, we have a couple of extra sets of them: $0\leq d_1<\ldots<d_r\leq1$ and $0\leq e_1<\ldots<e_t\leq1$, for which $\calU$ is determined by the former and $J$ is discretized by the quadrature rule associated to the latter. That is, $\calU\colon[0,h]\to\bbR$ is a polynomial of order $r-1$ determined by $r$ points $\bar U_i=\calU(d_ih)$, $i=1,\ldots,r$, and for which $U_j:=\calU(c_jh)$, $j=1,\ldots,s$, and $\hat U_k:=\calU(e_kh)$, $k=1,\ldots,t$, are mere evaluations. And the cost function $J$ is discretized by the formula $h \sum_{k=0}^{N-1}  \sum_{i=1}^t \hat b_i C(\hat Q^k_i,\hat P^k_i,\hat U^k_i)$, where with a similar notation the ``hat'' stands for weights and evaluations related to the $e$'s. Now, different cases arise:
\begin{itemize}
\item If $r>t$, one does an underevaluation of the controls within the discrete cost functional which will allow these to diverge (for instance a control could not appear explicitly in the discrete cost and therefore take any arbitrary value).
\item If on the contrary $r<t$, one does an overevaluation of the controls which, in the case of a coercive discrete cost function, will only increase the computational cost with, a priori, no convergence benefits.
\item Therefore, the case $r=t$ seems to be the best fit, which is the case where there is a minimal number of evaluations of the controls (each control is evaluated just once in the discrete cost) so to have convergence under the assumption of coercivity.
\end{itemize}
Assuming the last situation and continuing with the discussion, further cases arise:
\begin{itemize}
\item On the one hand, if $r>s$, the convergence rate of the controls will be limited by the accuracy of the discrete dynamics.
\item On the other hand, if $r<s$, the convergence rate of the controls will suffer from a computational payload by the high accuracy of the dynamics.
\item Therefore, the case $r=s$ seems again to be the best fit, which is the case that equates accuracy of the dynamics with convergence rate of the controls.
\end{itemize}
Finally, under the assumption $r=s=t$, choosing a unique set of collocation points $0\leq c_1<\ldots<c_s\leq1$ (we drop the ``hats'', ``bars'', $e$'s and $d$'s), one minimizes the number of polynomial evaluations and therefore the total computational cost (besides of simplifying the problem).
\end{remark}

In the following example, we solve a simple optimal control problem with a linear dynamical constraint and a quadratic cost function. The numerical experiments show, in the spirit of \cite{Ha00} and the previous discussion before it, how a good choice and proper combination of the discretization gives a ``fast'' convergence of the scheme, while other combinations show ``slow'' convergence or even divergence of the controls, all of it exemplifying Theorem \ref{thm:control:conv}.

\begin{example}
Consider the problem
\begin{subequations} \label{eq:ocp:hager}
\begin{align}
\label{eq:ocp:hager:cost}
&  \min_{q,\dq,u}\int_0^T(\dq^2+u^2) \,\dt\\
\label{eq:ocp:hager:dyn}
&  \textrm{s.t.}\quad\ddot q = 1+u\,,\quad(q(0),\dq(0))=(0,0)
\end{align}
\end{subequations}
for which the functions
\[ q(t) = \frac{\cosh(t)-1}{\cosh(T)} \qquand u(t) = \frac{\cosh(t)}{\cosh(T)}-1 \]
are the unique solution. We identify from the forced Euler-Lagrange equation \eqref{eq:ocp:hager:dyn} the Lagrangian function $L(q,\dq) = \frac12\dq^2+q$ and the control force $F(q,\dq,u) = u$. The density cost function is obviously $C(q,\dq,u) = \dq^2+u^2$.

We discretize the mechanical system by using a symplectic Galerkin approach together with a Lobatto quadrature for $s=3$ points. We initially assume that the controls are also discretized by $r=3$ nodes. Then the right- hand side equations of \eqref{eq:sG_2} or \eqref{eq:docp:sG:c} are
\begin{subequations}
\label{eq:docp:hager}
\begin{eqnarray}
\dP_1 = & \frac{-p_0}{h/6} + \frac1h(4\dQ_1+2\dQ_2) & = 1+U_1 \,,\\
\dP_2 = & \phantom{\frac{-p_0}{h/6} +} \frac1h(-\dQ_1+\dQ_3) & = 1+U_2 \,,\\
\dP_3 = & \frac{p_1}{h/6} + \frac1h(2\dQ_2-4\dQ_3) & = 1+U_3 \,,
\end{eqnarray}
\end{subequations}
where the micro-veloticies $\dQ_i$ are given by the left equations of \eqref{eq:sG_2} or \eqref{eq:docp:sG:c}, which are in this particular case
\[
\left(\begin{array}{c}
\dQ_1\\ \dQ_2\\ \dQ_3
\end{array}\right)
= \frac1h
\left(\begin{array}{rrr}
-3 &  4 & -1\\
-1 &  0 &  1\\
 1 & -4 &  3
\end{array}\right)
\left(\begin{array}{c}
Q_1\\ Q_2\\ Q_3
\end{array}\right)
\]

For the cost function, we consider four different discretizations with Lobatto's quadrature rule for $t=1,2,3,4$ quadrature points. These are respectively
\begin{subequations} \label{eq:docp:hager:cntrl3}
\begin{eqnarray}
\label{eq:docp:hager:cntrl3:cost1}
C_d(q_h,p_h,u_h) &=& h\left(\dQ_2^2+U_2^2\right) \,,\\
\label{eq:docp:hager:cntrl3:cost2}
C_d(q_h,p_h,u_h) &=& \frac{h}2\left(\dQ_1^2+\dQ_3^2+U_1^2+U_3^2\right) \,,\\
\label{eq:docp:hager:cntrl3:cost3}
C_d(q_h,p_h,u_h) &=& \frac{h}6\left(\dQ_1^2+4\dQ_2^2+\dQ_3^2+U_1^2+4U_2^2+U_3^2\right) \,,\\
\label{eq:docp:hager:cntrl3:cost4}
C_d(q_h,p_h,u_h) &=& \frac{h}{12}\left((3-\sqrt5)\left(\dQ_1^2+(3-\sqrt5)\dQ_2^2+\dQ_3^2+\right)\right.\\
            & & \phantom{\frac{h}{12}\left(\right.}\left.+(1-1/\sqrt5)\left((\dQ_1+\dQ_2)^2+(\dQ_2+\dQ_3)^2\right)\right) \nonumber\\
            & & + \frac{h}{30}\left(2(U_1+U_2)^2+(U_1-U_3)^2+2(U_2+U_3)^2\right. \nonumber\\
            & & \phantom{+\frac{h}{30}\left(\right.}\left.+ U_1^2+12U_2^2+U_3^2\right) \,. \nonumber
\end{eqnarray}
\end{subequations}
The first two discretizations, Equations \eqref{eq:docp:hager:cntrl3:cost1} and \eqref{eq:docp:hager:cntrl3:cost2}, are clearly not coercive with respect to the controls ($U_1$ and $U_3$ are missing in \eqref{eq:docp:hager:cntrl3:cost1} and $U_2$ is missing in \eqref{eq:docp:hager:cntrl3:cost2}), which will be allowed to diverge (see Figures~\ref{fig:HagerD3U3J1} and \ref{fig:HagerD3U3J2}). Nonetheless the last two discretizations, Equations \eqref{eq:docp:hager:cntrl3:cost3} and \eqref{eq:docp:hager:cntrl3:cost4}, are indeed coercive, still one outperforms the other in terms of convergence (see Figures~\ref{fig:HagerD3U3J3} and \ref{fig:HagerD3U3J4}). The discrete cost function \eqref{eq:docp:hager:cntrl3:cost4}, besides of having a higher computational cost, shows a slower convergence rate. The discrete cost function \eqref{eq:docp:hager:cntrl3:cost3} corresponds to the method presented in Problem \eqref{eq:docp:sG} and Theorem \ref{thm:control:conv}.

\newcommand{\tU}{\widetilde U}
We continue by assuming that the approximated control $\calU$ is determined only by two points, that is
\[ \calU(t) = \tU_1 + t(\tU_3-\tU_1)  \]
(we note $\tU_3$ instead of $\tU_2$ to make the notation more appealing). The previous set of Equations \eqref{eq:docp:hager} and \eqref{eq:docp:hager:cntrl3} are then updated by merely substituting the controls by
\[ U_1 = \tU_1 \,,\ U_2 = \tfrac12(\tU_1+\tU_3) \,,\ \textrm{and}\ U_3 = \tU_3 \,,  \]
which leads to
\begin{subequations} \label{eq:docp:hager:cntrl2}
\begin{eqnarray}
\label{eq:docp:hager:cntrl2:cost1}
C_d(q_h,p_h,u_h) &=& h\left(\dQ_2^2+(\tU_1+\tU_3)^2/4\right) \,, \\
\label{eq:docp:hager:cntrl2:cost2}
C_d(q_h,p_h,u_h) &=& \frac{h}2\left(\dQ_1^2+\dQ_3^2+\tU_1^2+\tU_3^2\right) \,, \\
\label{eq:docp:hager:cntrl2:cost3}
C_d(q_h,p_h,u_h) &=& \frac{h}6\left(\dQ_1^2+4\dQ_2^2+\dQ_3^2+\tU_1^2+(\tU_1+\tU_3)^2+\tU_3^2\right) \,, \\
\label{eq:docp:hager:cntrl2:cost4}
C_d(q_h,p_h,u_h) &=& \frac{h}{12}\left((3-\sqrt5)\left(\dQ_1^2+(3-\sqrt5)\dQ_2^2+\dQ_3^2+\right)\right.\\
            & & \phantom{\frac{h}{12}\left(\right.}\left.+(1-1/\sqrt5)\left((\dQ_1+\dQ_2)^2+(\dQ_2+\dQ_3)^2\right)\right) \nonumber\\
            & & + \frac{h}{6}\left(\tU_1^2+(\tU_1+\tU_3)^2+\tU_3^2\right) \,. \nonumber
\end{eqnarray}
\end{subequations}
In this occasion, only the first discretization, Equation \eqref{eq:docp:hager:cntrl2:cost1}, defines a non-coercive discrete cost function (see Figure~\ref{fig:HagerD3U2J1}). From the rest (see Figures~\ref{fig:HagerD3U2J2}-\ref{fig:HagerD3U2J4}), Equations \eqref{eq:docp:hager:cntrl2:cost3} and \eqref{eq:docp:hager:cntrl2:cost4} show the fastest convergence rate, although slow on the controls and with a computational payload for \eqref{eq:docp:hager:cntrl2:cost4}. Equation \eqref{eq:docp:hager:cntrl2:cost3} corresponds to a discretization of the cost with three quadrature points.

\begin{figure}
   \centering
   \begin{subfigure}[b]{0.4\textwidth}
      \includegraphics[width=\textwidth]{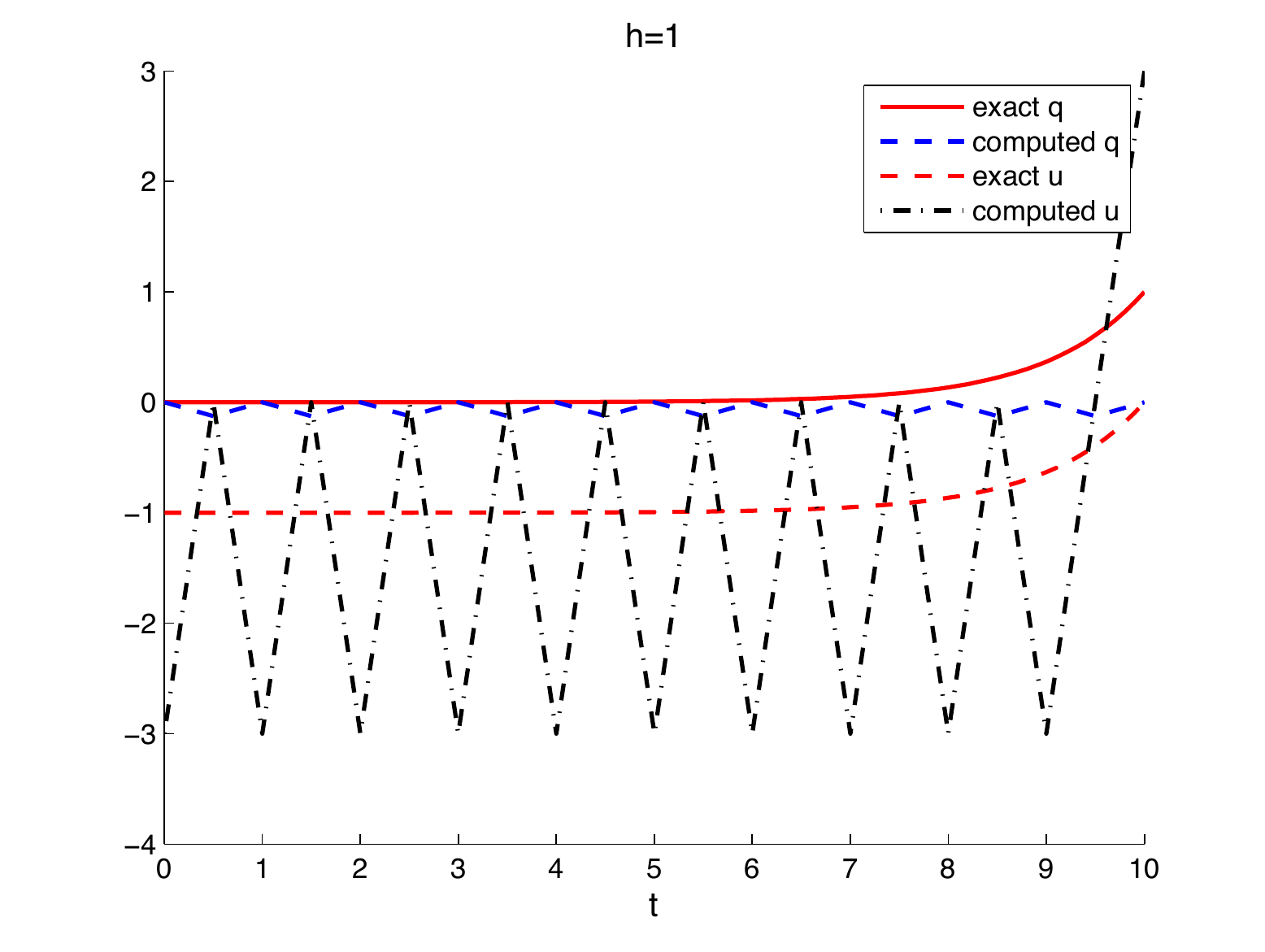}
      \caption{cost function \eqref{eq:docp:hager:cntrl3:cost1}}
      \label{fig:HagerD3U3J1}
   \end{subfigure}
   \hfill
   \begin{subfigure}[b]{0.4\textwidth}
      \includegraphics[width=\textwidth]{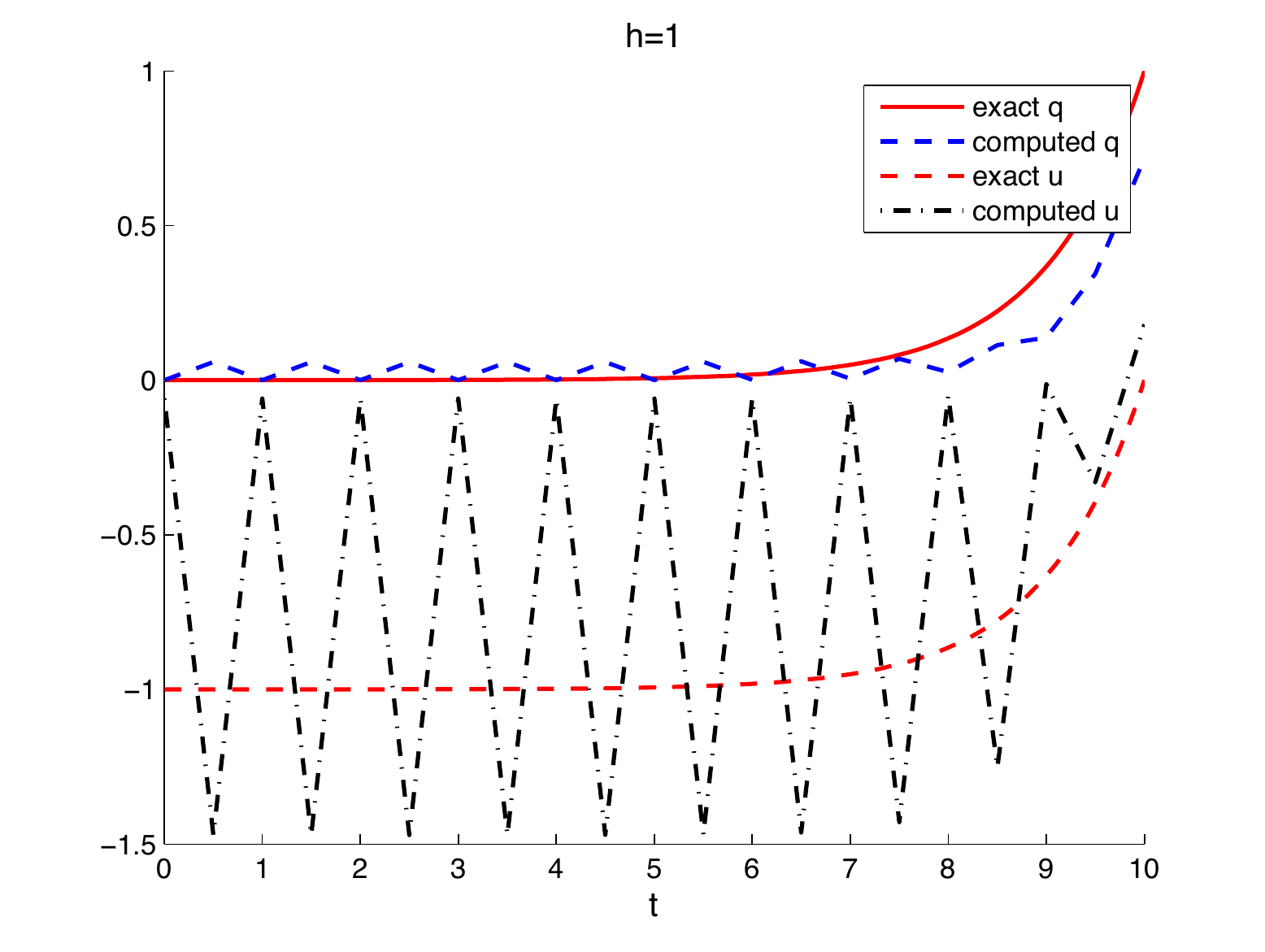}
      \caption{cost function \eqref{eq:docp:hager:cntrl3:cost2}}
      \label{fig:HagerD3U3J2}
   \end{subfigure}
   \\
   \begin{subfigure}[b]{0.4\textwidth}
      \includegraphics[width=\textwidth]{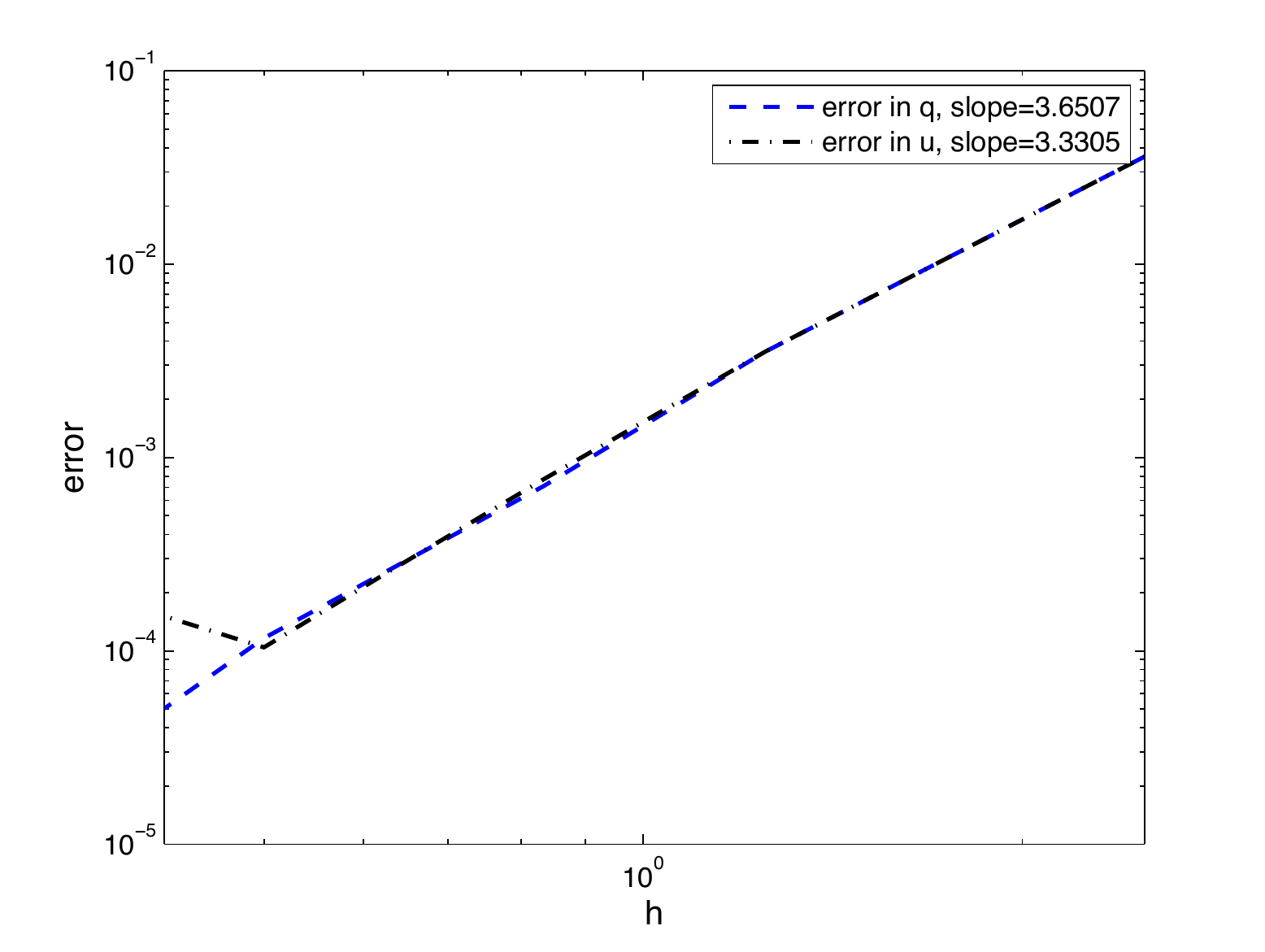}
      \caption{cost function \eqref{eq:docp:hager:cntrl3:cost3}}
      \label{fig:HagerD3U3J3}
   \end{subfigure}
   \hfill
   \begin{subfigure}[b]{0.4\textwidth}
      \includegraphics[width=\textwidth]{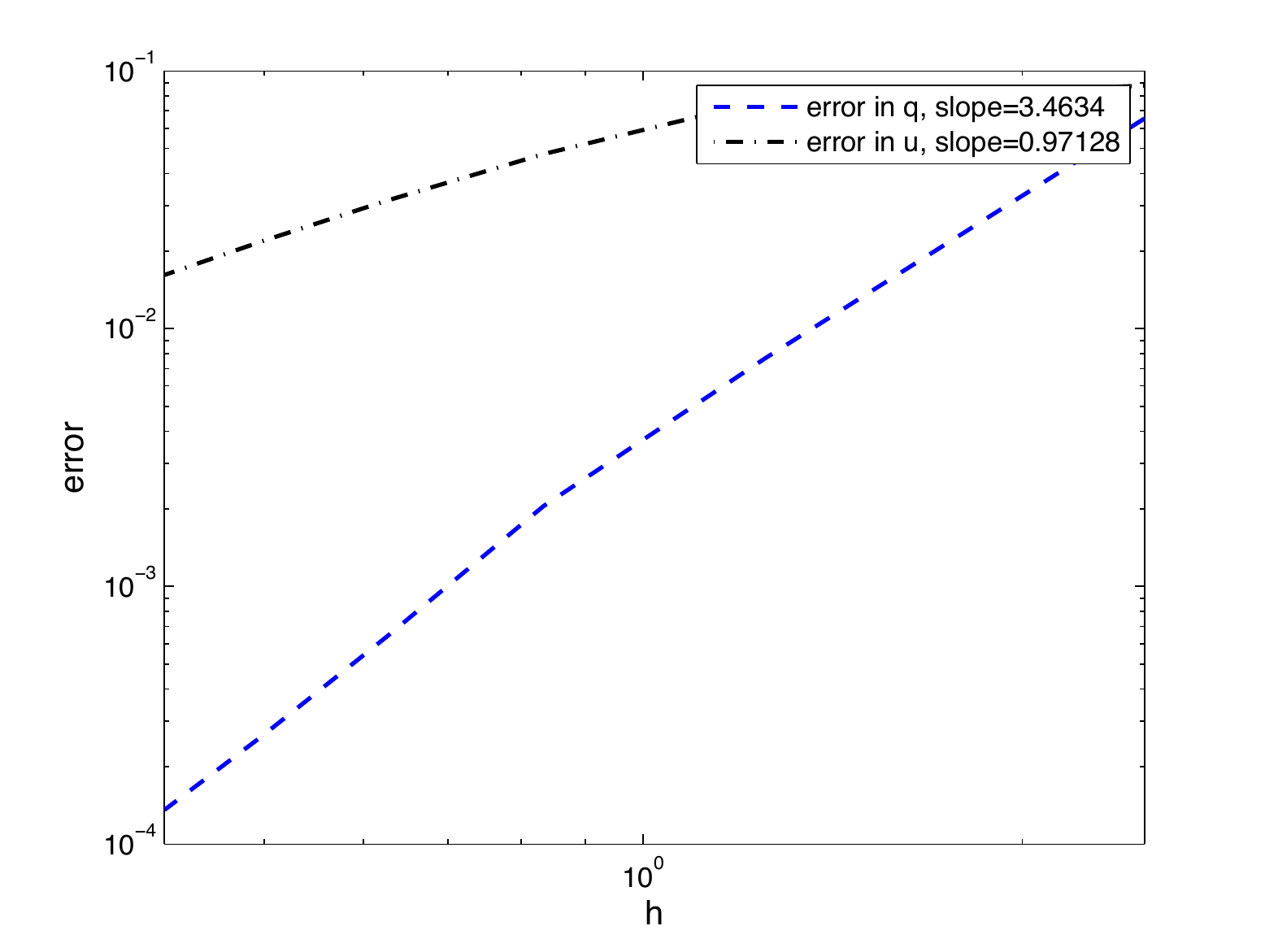}
      \caption{cost function \eqref{eq:docp:hager:cntrl3:cost4}}
      \label{fig:HagerD3U3J4}
   \end{subfigure}
   \caption{Convergence behavior of the discrete solution for the discrete cost functions given in \eqref{eq:docp:hager:cntrl3}.}
   \label{fig:HagerD3U3}
\end{figure}

\begin{figure}
   \centering
   \begin{subfigure}[b]{0.4\textwidth}
      \includegraphics[width=\textwidth]{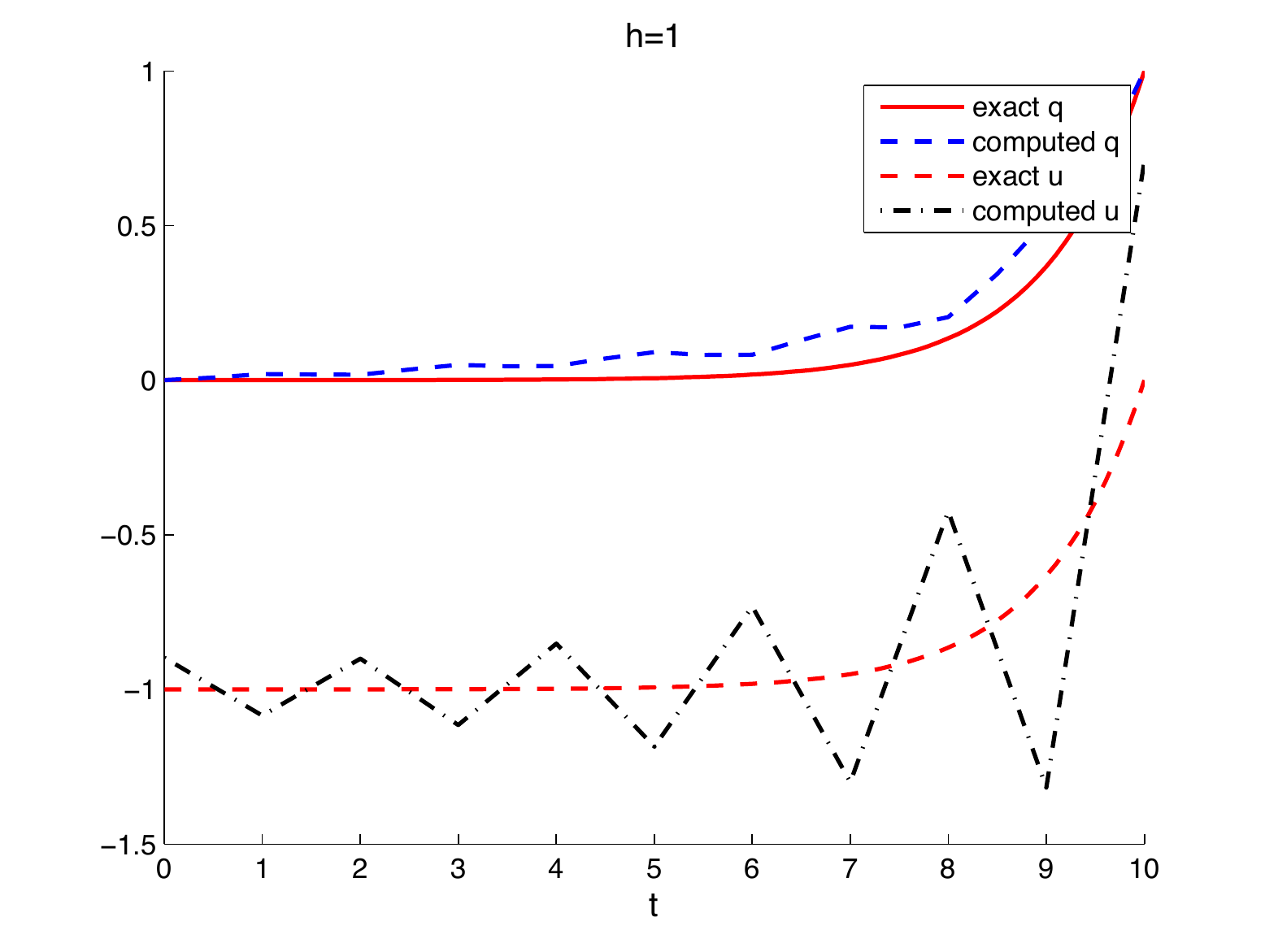}
      \caption{cost function \eqref{eq:docp:hager:cntrl2:cost1}}
      \label{fig:HagerD3U2J1}
   \end{subfigure}
   \hfill
   \begin{subfigure}[b]{0.4\textwidth}
      \includegraphics[width=\textwidth]{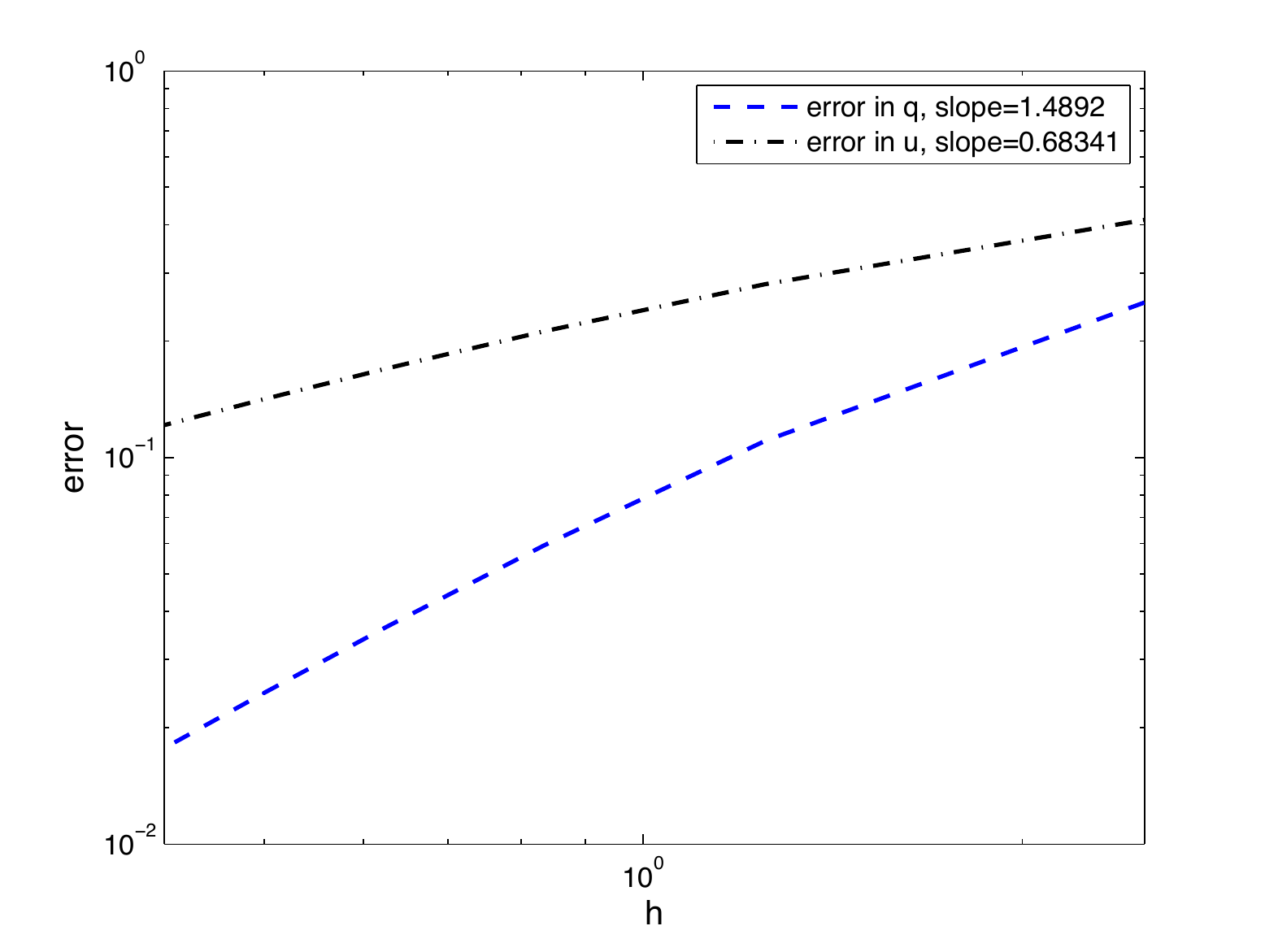}
      \caption{cost function \eqref{eq:docp:hager:cntrl2:cost2}}
      \label{fig:HagerD3U2J2}
   \end{subfigure}
   \\
   \begin{subfigure}[b]{0.4\textwidth}
      \includegraphics[width=\textwidth]{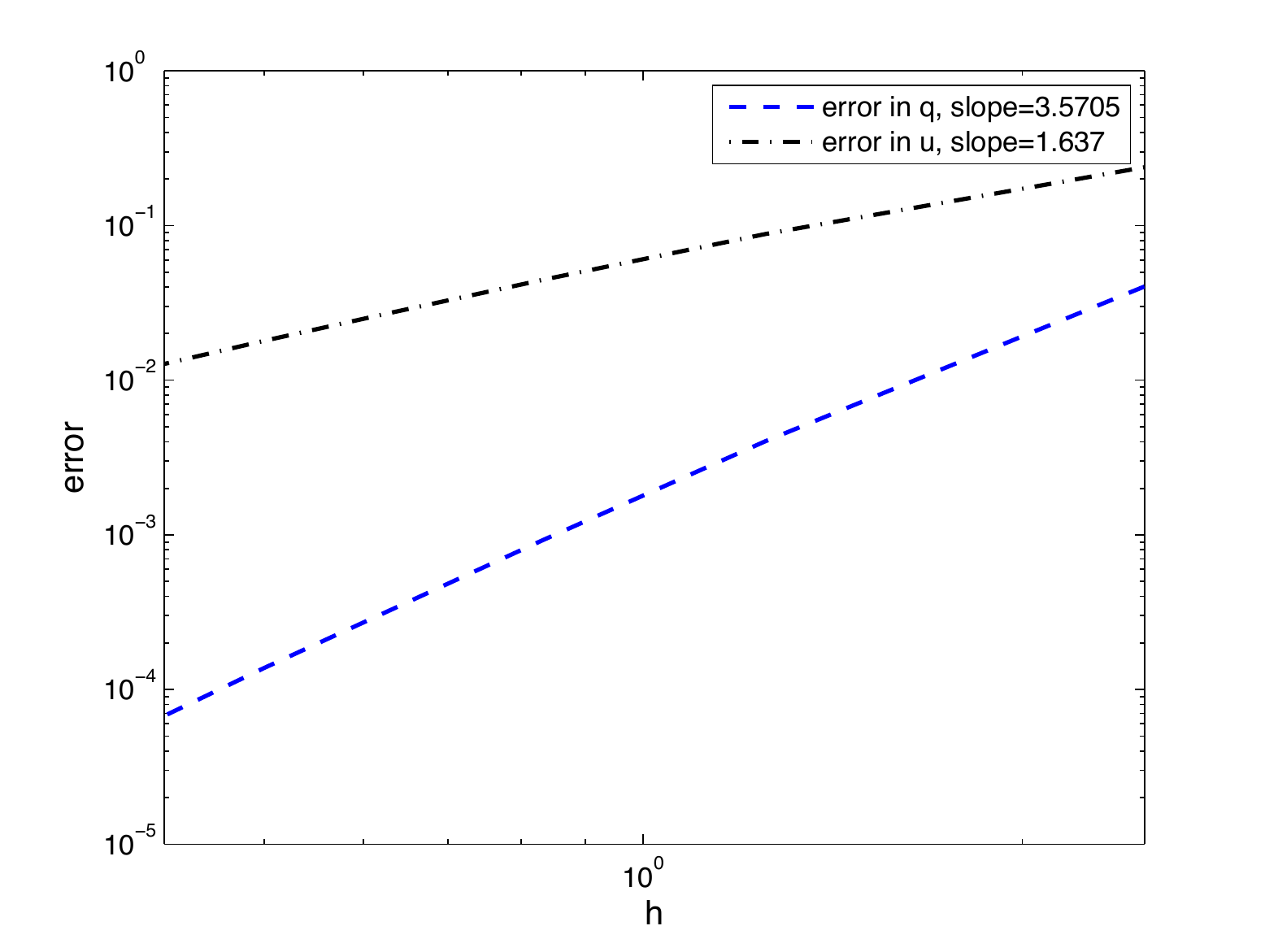}
      \caption{cost function \eqref{eq:docp:hager:cntrl2:cost3}}
      \label{fig:HagerD3U2J3}
   \end{subfigure}
   \hfill
   \begin{subfigure}[b]{0.4\textwidth}
      \includegraphics[width=\textwidth]{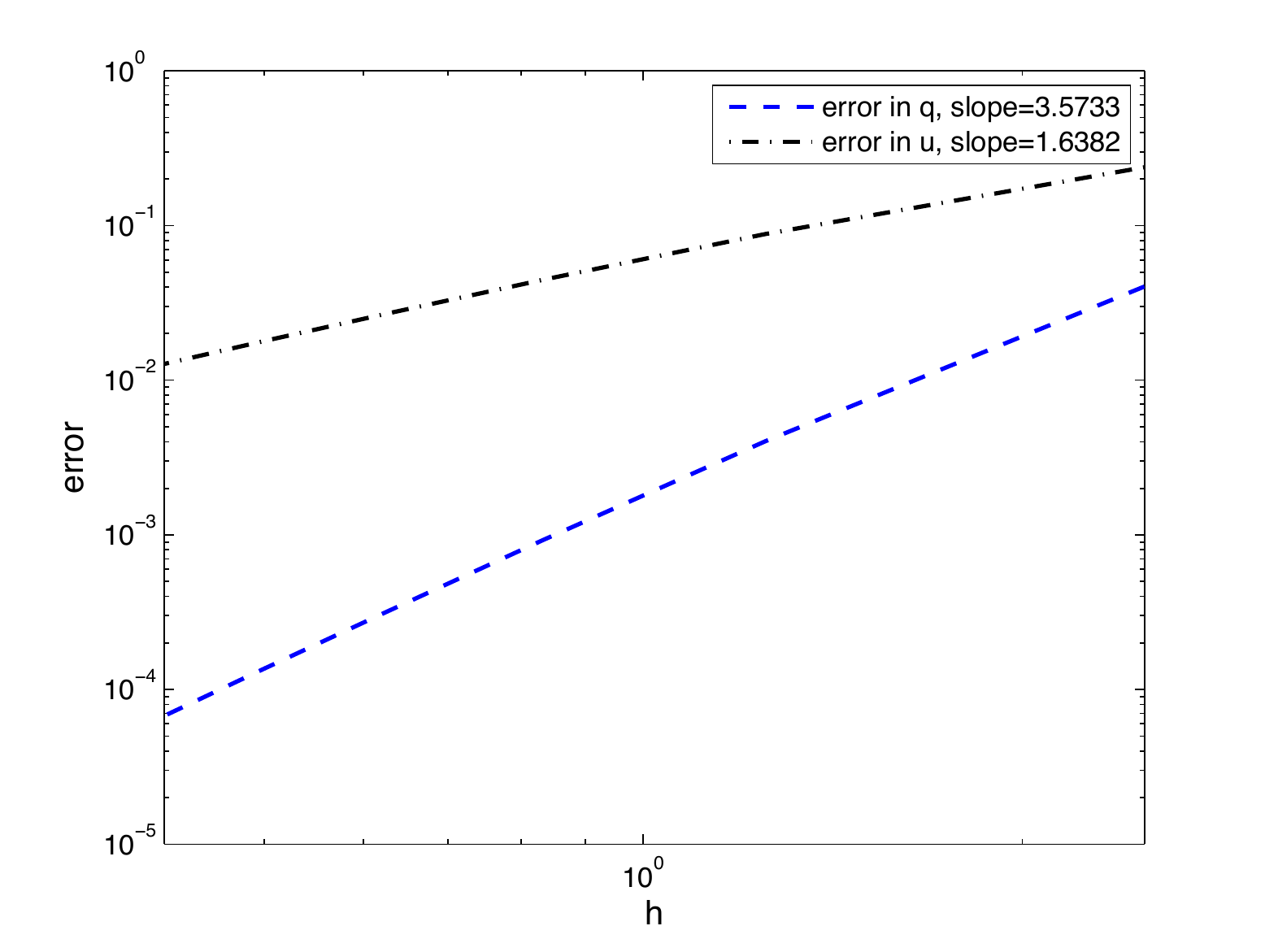}
      \caption{cost function \eqref{eq:docp:hager:cntrl2:cost4}}
      \label{fig:HagerD3U2J4}
   \end{subfigure}
   \caption{Convergence behavior of the discrete solution for the discrete cost functions given in \eqref{eq:docp:hager:cntrl2}.}
   \label{fig:HagerD3U2}
\end{figure}
\end{example}

\section{Commutation of discretization and dualization} \label{sec:conmutation}
In this section we investigate the equivalence (C) in Figure~\ref{fig:intro_comparison} for the special choice of sG discretization of the optimal control problem~\ref{prob:ocp}. To this end, we analyze and compare the adjoint systems for the continuous and the discrete optimal control problems.

Throughout the section we assume that all controls under consideration do not saturate the constraints. In other words, we assume that the optimal control is in the interior of the set of constraints on controls. This is obviously the case if $U=\bbR^m$. This assumption allows us to avoid the typical situation of bang-bang controls, and under slight extra conditions, to derive (from the Pontryagin Maximum Principle) extremal controls that are smooth functions of the state and costate. The necessary optimality conditions \eqref{eq:minprinc} are
\begin{subequations} \label{eq:general:ocp}
\begin{align}
\label{eq:general:ocp_1}
\dot\lambda &= -\nabla_qC-\lambda\cdot\nabla_qf-\psi\cdot\nabla_qg\,, &\lambda(T)&=\nabla_q\Phi(q(T),p(T))\,,\\
\label{eq:general:ocp_2}
\dot\psi &=  -\nabla_pC-\lambda\cdot\nabla_pf-\psi\cdot\nabla_pg\,,  &\psi(T) &= \nabla_q\Phi(q(T),p(T))\,,\\
\label{eq:general:ocp_3}
0 &= \nabla_uC +  \psi\cdot\nabla_ug\,.
\end{align}
\end{subequations}
If we use the sG integrator for the discretization of Problem~\ref{prob:ocp}, we obtain the discretized optimal control problem~\ref{prob:docp:sG}. To derive the necessary optimality conditions for the discretized optimal control problem, we introduce the discrete adjoint vectors (covectors in $\bbR^n$) $\lambda_0,\ldots,\lambda_N$, $\mu_0,\ldots,\mu_{N-1}$, $\psi_0$, $\Lambda_i^0,\ldots,\Lambda_i^{N-1}$, $\Psi_i^0,\ldots,\Psi_i^{N-1}$, $i=1,\ldots,s$, and define the discrete optimal control Lagrangian as
\begin{align}
\mathcal{L}_d =\,&  \sum_{k=0}^{N-1} \sum_{i=1}^s h b_i C_i^k +   \Phi(q_N,p_N) - \lambda_0\cdot (q_0-q^0) - \psi_0\cdot (p_0-p^0)  \nonumber\\
& + \sum_{k=0}^{N-1}  \left[ \mu_k\cdot \left(q_k - \sum_{j=1}^s\alpha^jQ_j^k\right) - \lambda_{k+1}\cdot \left(q_{k+1} - \sum_{j=1}^s\beta^jQ_j^k\right)\right. \nonumber\\
& +   \sum_{i=1}^s \Lambda_i^k\cdot \left( h f_i^k -  \sum_{j=1}^s a_{ij}Q_j^k\right) \nonumber\\
& + \left.  \Psi_i^k\cdot \left( h g_i^k -  \frac{\beta^ip_{k+1}-\alpha^ip_k}{\bar b_i} - \sum_{j=1}^s\bar a_{ij}P_j^k \right)\right], \label{eq:ocpLd}
\end{align}
where $C_i^k$ is a short notation for $C(Q_i^k,P_i^k,U_i^k)$ (analogously for $f_i^k$ and $g_i^k$).
The necessary optimality conditions (KKT equations) are derived by differentiation w.r.t.~the discrete variables $q_k,p_k$, $k=0,\ldots,N$ and $Q_i^k, P_i^k, U_i^k$, $k=0,\ldots,N-1$, $i=1,\ldots,s$,
which leads to
\begin{subequations}\label{eq:sG_adjoint_or}
\begin{align}
\text{for}\; k=0,\ldots,N-1:\qquad  \mu_k -\lambda_k &=0,\label{eq:mulambda}\\
 \nabla_q\Phi(q_N,p_N) - \lambda_N  &=0, \label{eq:sG_adjoint_or_2}\\
 -\psi_0 + \sum_{i=1}^s \frac{\alpha^i}{\bar{b}_i}\Psi_i^0&=0, \label{eq:sG_adjoint_or_3}\\
 \text{for}\; k=1,\ldots,N-1:\qquad \sum_{i=1}^s \frac{\alpha_i}{\bar b_i}\Psi_i^k  - \sum_{i=1}^s \frac{\beta^i}{\bar b_i}\Psi_i^{k-1} &=0, \label{eq:sG_adjoint_or_4}\\
\nabla_p\Phi(q_N,p_N)- \sum_{i=1}^s \frac{\beta^i}{\bar{b}_i}\Psi_i^{N-1} &=0, \label{eq:sG_adjoint_or_5}\\
\frac1h\bigg(-\alpha^i\mu_k + \beta^i\lambda_{k+1} - \sum_{j=1}^sa_{ji}\Lambda_j^k\bigg) + b_i\nabla_qC_i^k + \Lambda_i^k\cdot\nabla_qf_i^k + \Psi_i^k\cdot\nabla_qg_i^k&=0, \label{eq:sG_adjoint_or_6}\\
-\frac1h\sum_{j=1}^s\bar{a}_{ji}\Psi_j^k + b_i\nabla_pC_i^k + \Lambda_i^k\cdot\nabla_pf_i^k + \Psi_i^k\cdot\nabla_pg_i^k &=0, \label{eq:sG_adjoint_or_7}\\
  b_i\nabla_uC_i^k + \Psi_i^k\cdot\nabla_ug_i^k &= 0, \label{eq:sG_adjoint_or_8}
\end{align}
with $k=0,\ldots,N-1, i=1,\ldots,s$, for the last three equations.
\end{subequations}
\begin{subequations}
We transform the necessary optimality conditions by defining
\begin{gather}
\label{eq:defgammchi}
\Gamma_i^k := \Lambda_i^k/\bar{b}_i \quand \chi_i^k := \Psi_i^k/\bar{b}_i \quad\text{for}\quad k=0,\ldots,N-1,\, i=1,\ldots,s,\\
\label{eq:psi+-}
\psi_k^- := \sum_{i=1}^s\alpha^i\chi_i^k,\, k=0,\ldots,N-1,\quand  \psi_k^+  := \sum_{i=1}^s\beta^i\chi_i^{k-1} ,\, k=1,\ldots,N,
\end{gather}
such that Equation~\eqref{eq:sG_adjoint_or_4} reduces to $\psi_k^- = \psi_k^+ :=\psi_k$.
\end{subequations}
By eliminating the variables $\mu_0,\ldots,\mu_{N-1}$ with Equation~\eqref{eq:mulambda} and by exploiting the conditions on the coefficients $b_ia_{ij}+\bar b_j\bar a_{ji}=0$ and $b_i = \bar{b}_i$, we obtain the following discrete adjoint system
\begin{subequations} \label{eq:sG_adjointgeneral}
\begin{align}
\label{eq:sG_adjointgeneral_1}
\psi_k =& \sum_{j=1}^s\alpha^j\chi_j^k\,,\\
\label{eq:sG_adjointgeneral_2}
\psi_{k+1} =& \sum_{j=1}^s\beta^j\chi_j^k\,,\\
\label{eq:sG_adjointgeneral_3}
-\nabla_qC_i^k-\Gamma_i^k\cdot\nabla_qf_i^k-\chi_i^k\cdot\nabla_qg_i^k =& \frac{\beta^i\lambda_{k+1}-\alpha^i\lambda_k}{h\bar b_i} + \frac1h\sum_{j=1}^s\bar a_{ij}\Gamma_j^k\,,\\
\label{eq:sG_adjointgeneral_4}
-\nabla_pC_i^k-\Gamma_i^k\cdot\nabla_pf-\chi_i^k\cdot\nabla_pg_i^k =& \frac1h\sum_{j=1}^sa_{ij}\chi_j^k\,,\\
\label{eq:sG_adjointgeneral_5}
\nabla_uC_i^k+\chi_i^k\cdot\nabla_ug_i^k =& 0\,,
\end{align}
\text{for $k=0,\ldots,N-1,\, i=1,\ldots,s$, and with final conditions}
\begin{equation}
\label{eq:sG_adjointgeneral_6}
\lambda_N  = \nabla_q\Phi(q_N,p_N) \quand \psi_N= \nabla_p\Phi(q_N,p_N)\,,
\end{equation}
where
\begin{equation}
\label{eq:coeffab}
b_ia_{ij}+\bar b_j\bar a_{ji}=0 \quand b_i=\bar b_j\,.
\end{equation}
\end{subequations}
Note that the adjoint scheme~\eqref{eq:sG_adjointgeneral_1}-\eqref{eq:sG_adjointgeneral_4} together with the final constraints~\eqref{eq:sG_adjointgeneral_6} and the conditions on the coefficients \eqref{eq:coeffab} is exactly the symplectic Galerkin integrator applied to the adjoint system~\eqref{eq:general:ocp_1}-\eqref{eq:general:ocp_2}. To ensure that the discrete adjoint system~\eqref{eq:sG_adjointgeneral} is indeed equivalent to the necessary optimality conditions defined in \eqref{eq:sG_adjoint_or} we show the following proposition.

\begin{proposition}
\newcommand{\brk}{\allowbreak}
If $\bar{b}_i>0$ for each $i$, then the necessary optimality conditions~\eqref{eq:sG_adjoint_or} and the discrete adjoint system~\eqref{eq:sG_adjointgeneral} are equivalent.
That is, if $(\mu_0,\ldots,\mu_{N-1},\brk\Lambda_i^0,\ldots,\Lambda_i^{N-1},\brk\Psi_i^0,\ldots,\Psi_i^{N-1})$, $i=1,\ldots,s$, satisfy \eqref{eq:sG_adjoint_or}, then \eqref{eq:sG_adjointgeneral} hold for $(\psi_k,\Gamma_i^k,\brk\chi_i^k)$ defined in \eqref{eq:defgammchi} and \eqref{eq:psi+-}. Conversely, if $(\psi_1,\ldots,\psi_N,\brk\Gamma_i^0,\ldots,\Gamma_i^{N-1},\brk\chi_i^0,\ldots,\brk\chi_i^{N-1})$, $i=1,\ldots,s$, satisfy \eqref{eq:sG_adjointgeneral}, then \eqref{eq:sG_adjoint_or} hold for $(\mu_k,\Lambda_i^k,\Psi_i^k)$ defined in \eqref{eq:mulambda} and \eqref{eq:defgammchi}.
\end{proposition}

\begin{proof}
\newcommand{\brk}{\allowbreak}
We already derived the adjoint system~\eqref{eq:sG_adjointgeneral} starting from the necessary optimality conditions~\eqref{eq:sG_adjoint_or}. We now suppose that $(\psi_1,\ldots,\psi_N,\brk\Gamma_i^0,\ldots,\Gamma_i^{N-1},\brk\chi_i^0,\ldots,\brk\chi_i^{N-1})$, $i = 1,\ldots,s$, satisfy the adjoint system~\eqref{eq:sG_adjointgeneral}. Equation~\eqref{eq:mulambda} holds by assumption. The condition for $\lambda_N$ in \eqref{eq:sG_adjointgeneral_6} and \eqref{eq:sG_adjoint_or_2} are identical. The condition for $\psi_N$ in \eqref{eq:sG_adjointgeneral_6} together with Equation~\eqref{eq:sG_adjointgeneral_2} for $k=N-1$ and the definition~\eqref{eq:defgammchi} yields \eqref{eq:sG_adjoint_or_5} whereas Equation~\eqref{eq:sG_adjointgeneral_1} for $k=0$ together with definition~\eqref{eq:defgammchi} yields \eqref{eq:sG_adjoint_or_3}. By subtracting Equations~\eqref{eq:sG_adjointgeneral_1} and \eqref{eq:sG_adjointgeneral_2} for the same index $k$ and using definition~\eqref{eq:defgammchi} we obtain \eqref{eq:sG_adjoint_or_4}. Finally, by taking the condition~\eqref{eq:coeffab} on the coefficients into account, \eqref{eq:sG_adjointgeneral_3}-\eqref{eq:sG_adjointgeneral_5} and definition \eqref{eq:defgammchi} yield \eqref{eq:sG_adjoint_or_6}-\eqref{eq:sG_adjoint_or_8}, respectively.
\end{proof}

With the classical Legendre assumption, \ie\  $(\pp*[^2\calH]{u^2})(q^*,p^*,u^*,\lambda,\psi,1)$ is a positive definite symmetric matrix, with Equation~\eqref{eq:general:ocp_3} $u$ can be expressed as function of the states and the adjoints, $u=u(q,p,\lambda,\psi)$. We denote by $\nu$ and $\eta$ the functions defined by
\begin{align*}
 \nu(q,p,\lambda,\psi) &= \left(-\nabla_qC(q,p,u)-\lambda\cdot\nabla_qf(q,p)-\psi\cdot\nabla_qg(q,p,u)\right)|_{u=u(q,p,\lambda,\psi)}\,,\\
\eta(q,p,\lambda,\psi) &= \left(-\nabla_pC(q,p,u)-\lambda\cdot\nabla_pf(q,p)-\psi\cdot\nabla_pg(q,p,u)\right)|_{u=u(q,p,\lambda,\psi)}\,.
\end{align*}
With some abuse of notation, let $g(q,p,\lambda,\psi)$ denote the function $g(q,p,u(q,p,\lambda,\psi))$.
In the case where the control has the form $U_i^k = u(Q_i^k,P_i^k,\Gamma_i^k,\chi_i^k)$, the state and adjoint scheme based on the symplectic Galerkin integrator can be expressed as
\begin{subequations}\label{eq:state_adjoint_discrete}
\begin{align}
  q_k &= \sum_{j=1}^s\alpha^jQ^k_j\,,\quad  q_{k+1} = \sum_{j=1}^s\beta^jQ^k_j\,, \quad \psi_k = \sum_{j=1}^s\alpha^j\chi_j^k\,,\quad   \psi_{k+1} = \sum_{j=1}^s\beta^j\chi_j^k\,,\\
f^k_i &= \frac{1}{h}\sum_{j=1}^sa_{ij}Q^k_j\,,\quad  g^k_i = \frac{\beta^ip_{k+1}-\alpha^ip_k}{h\bar b_i} + \frac1h\sum_{j=1}^s\bar a_{ij}P^k_j\,,\\
\eta^k_i &= \frac{1}{h}\sum_{j=1}^sa_{ij}\chi_j^k\,, \quad \nu^k_i   = \frac{\beta^i \lambda_{k+1}-\alpha^i \lambda_k}{h\bar b_i} + \frac1h\sum_{j=1}^s\bar a_{ij}\Gamma_j^k\,,
\end{align}
$k=0,\ldots,N-1,\, i=1,\ldots,s$,
\begin{equation}
q_0=q^0,\quad p_0 = p^0,\quad \lambda_N  = \nabla_q\Phi(q_N,p_N), \quad \psi_N= \nabla_p\Phi(q_N,p_N),
\end{equation}
\end{subequations}
where $b_ia_{ij}+\bar b_j\bar a_{ji}=0$ and $b_i=\bar b_j$ and where $f_i^k$ and $g^k_i$ are short notations for $f(Q_i^k,P_i^k)$ and $g(Q^k_i,P^k_i,\Gamma_i^k,\chi_i^k)$ (analogously for $\eta_i^k$ and $\nu^k_i$).

Scheme~\eqref{eq:state_adjoint_discrete} can be viewed as symplectic Galerkin discretization of the two-point boundary value problem
\begin{subequations}\label{eq:state_adjoint}
\begin{align}
\dq &= f(q,p) \,,& q(0) &= q^0 \,, \\
\dot p &= g(q,p,\lambda,\psi) \,,& p(0) &= p^0 \,,\\
\dot\lambda &= \nu(q,p,\lambda,\psi)  \,,& \lambda(T) &= \nabla_q\Phi(q(T),p(T)) \,,\\
\dot\psi &= \eta(q,p,\lambda,\psi) \,,& \psi(T) &= \nabla_q\Phi(q(T),p(T)) \,,
\end{align}
\end{subequations}
where the variables $(q,\psi)$ and $(p,\lambda)$ are treated in the same way, respectively. Since the same discrete scheme is used for state and adjoint system, the orders of approximation coincide. This leads to the following statement.

\begin{theorem}[Commutation property]\label{th:commutation}
Given the $(OCP)$ \ref{prob:ocp}, besides of \ref{thm:control:conv:regularity}, \ref{thm:control:conv:coercivity}, \ref{thm:control:conv:uniqueness} from Theorem~\ref{thm:control:conv}, we assume that $(\pp*[^2\calH]{u^2})(q^*,p^*,u^*,\lambda,\psi,1)$ is a positive definite symmetric matrix. If a convergent symplectic Galerkin method with $b_i>0,\, i=1,\ldots,s,$ is used for the discretization of the state system, dualization and discretization commute, \ie\ the dualization of Problem~\ref{prob:docp:sG} coincides with the sG discretization of the boundary value problem~\eqref{eq:state_adjoint}.
\end{theorem}

\begin{remark}
In spirit of the Covector Mapping Principle (see \cite{GoRoKaFa08}), the order-preserving map between the adjoint variables corresponding to the dualized discrete problem (KKT) and the discretized dual problem (discrete PMP) is given by Equation~\eqref{eq:defgammchi}.
\end{remark}

\begin{remark}
If the controls do not saturate the constraints, \ie\ control constraints are active, the optimal solution is
typically only Lipschitz continuous. Then we expect
analogously to \cite{DoHa00} that convergence rates are limited to order two even for higher
order approximation schemes.
\end{remark}

\section{Conclusions}\label{sec:conclusion}
In this work, we investigate the application of high order variational integrators to the numerical solution of optimal control problems of mechanical systems. We derive two different schemes of high order variational integrators, the spRK and the sG method, which are both used for the discretization of a Lagrangian optimal control problem. The convergence of the primal variables of the resulting discrete optimal control problems is proven. Furthermore, the commutation of dualization and discretization for the sG method is shown, which extends the result in \cite{ObJuMa10} to another class of variational integrators that fulfills this commutation property and directly implies that the Covector Mapping Principle is satisfied. In particular, due to the commutation, not only the order of the adjoint scheme but also the discretization method itself is preserved and in contrast to Legendre pseudospectral methods or classical Runge-Kutta methods, no additional closure conditions (see \cite{GoRoKaFa08}) or conditions on the Runge Kutta coefficients (see \cite{Ha00}), respectively, are required.

The fulfillment of the Covector Mapping Principle provides a convenient way to prove the convergence of the dual variables (as done, for example, in \cite{Ha00}). With Theorem~\ref{th:commutation} the solution of the discretized direct problem coincides with the discrete solution of a shooting method applied to the necessary conditions of optimality. By showing the convergence of the shooting approach, we can conclude directly the convergence of the direct approach. The convergence proof for the adjoint variables is left for future work.

In the present paper we restricted ourselves to optimal control problems without any constraint on the final state. If we consider more general optimal control problems, involving constraints on the final state, then we expect that we will have to use the general conjugate point theory (see \cite{BoCaTr07}), in order to provide second-order conditions for optimality, related with the classical sensitivity analysis along a given optimal trajectory. Also, if there are some constraints on the final point then abnormal extremals may occur in the application of the Pontryagin Maximum Principle, which may raise a major problem in the analysis. Fortunately, it is known that abnormal minimizers do not exist under generic assumptions on the system and on the cost (see \cite{ChJeTr08}), and we expect that the results presented in this paper may hold in such a generic context. Otherwise the possible presence of abnormal minimizers is responsible for a loss of compactness (in particular, adjoint vectors do not stay in a compact anymore, see \cite{Tr00}), which may imply the failure of our method of proof, and it is not very clear then if we can expect that the Covector Mapping Principle hold true in that case.
These general considerations will be investigated in future work.
We stress again that, in the present paper, we have restricted ourselves to a more simple and tractable case. Once again, note that we have considered control-affine systems with (quasi)-quadratic costs, with assumptions ensuring the smoothness of optimal controls (as functions of the state and of the costate). As in \cite{Ha00}, our theory can be applied whenever there exist control constraints, provided the optimal controls under consideration belong to the interior of the set of constraints. Otherwise, in the general case controls may saturate the constraints, typically bang-bang controls do appear and then additional assumptions must be done in order to ensure a nice regularity of controls as functions of the state and of the costate. For instance, it is desirable to avoid chattering phenomena, in which a given optimal bang-bang control can have an infinite number of commutations over a compact interval. Also, we expect that one has to use the corresponding conjugate point theory in the bang-bang case. Such a theory does exist in the purely bang-bang case (see \cite{AgStZe98,MaOs04}) but then requires additional assumptions (ruling out, in particular, chattering). It can be noted that a general conjugate point theory, involving all possible subarcs -- free, bang, singular, boundary -- is still to be done. In order to get a general Covector Mapping Principle, such a complete theory is certainly required.

Another interesting issue is the application of more general symplectic and structure preserving methods in optimal control. Whereas the commutation property is shown for an already rich class of symplectic integrators (spRK and sG) it is still an open question, if it is satisfied for any variational and thus, any symplectic integrator (note that the classes of variational and symplectic integrators are identical, see \eg\ \cite{HaLuWa02}).

\section*{Acknowledgments}
This work has been partially supported by the European Union under the 7th Framework Programme FP7--PEOPLE--2010--ITN, grant agreement number 264735--SADCO, the Spanish MINECO under the National Reasearch Programme I+D+i 2010--2013, grant agreement number MTM2010-21186-C02-01. One of the authors, C.M.C., thanks the Universidad de Valladolid for its current postdoc position cofunded by the European Social Fund and the Junta de Castilla y León.

\bibliographystyle{siam}
\bibliography{hovi}

\def\cprime{$'$}
\begin{thebibliography}{10}

\bibitem{AbMa78}
{\sc R.~Abraham and J.~E. Marsden}, {\em Foundations of mechanics},
  Benjamin/Cummings Publishing Co. Inc. Advanced Book Program, Reading, Mass.,
  1978.
\newblock Second edition, revised and enlarged, With the assistance of Tudor
  Ra{\c{t}}iu and Richard Cushman.

\bibitem{AgStZe98}
{\sc A.~Agrachev, G.~Stefani, and Z.~P.}, {\em A hamitonian approach to strong
  minima in optimal control}, in Proceedings of Symposia in Pure Mathematics,
  vol.~64, American Mathematical Society, 1998, pp.~11--22.

\bibitem{Betts98}
{\sc J.~T. Betts}, {\em Survey of numerical methods for trajectory
  optimization}, J. Guid. Contr. Dynam., 21 (1998), pp.~193--207.

\bibitem{Bo06}
{\sc J.~F. Bonnans and J.~Laurent-Varin}, {\em Computation of order conditions
  for symplectic partitioned {R}unge-{K}utta schemes with application to
  optimal control}, Numer. Math., 103 (2006), pp.~1--10.

\bibitem{BoCaTr07}
{\sc B.~Bonnard, J.-B. Caillau, and E.~Tr{\'e}lat}, {\em Second order
  optimality conditions in the smooth case and applications in optimal
  control}, ESAIM Control Optim. Calc. Var., 13 (2007), pp.~207--236
  (electronic).

\bibitem{Ca13}
{\sc C.~M. Campos}, {\em High order variational integrators: a polynomial
  approach}, in Advances in Differential Equations and Applications, F.~Casas
  and V.~Martínez, eds., vol.~4 of SEMA SIMAI Springer Series, Springer
  International Publishing, 2014, pp.~3--49.

\bibitem{CaCeDi12}
{\sc C.~M. Campos, H.~Cendra, V.~A. D{\'{\i}}az, and D.~Mart{\'{\i}}n~de
  Diego}, {\em Discrete {L}agrange-d'{A}lembert-{P}oincar\'e equations for
  {E}uler's disk}, Rev. R. Acad. Cienc. Exactas F\'\i s. Nat. Ser. A Math.
  RACSAM, 106 (2012), pp.~225--234.

\bibitem{CaJuOb12}
{\sc C.~M. Campos, O.~Junge, and S.~Ober-Bl{\"o}baum}, {\em Higher order
  variational time discretization of optimal control problems}, in 20th
  International Symposium on Mathematical Theory of Networks and Systems,
  Melbourne, Australia, 2012.

\bibitem{ChJeTr08}
{\sc Y.~Chitour, F.~Jean, and E.~Tr{\'e}lat}, {\em Singular trajectories of
  control-affine systems}, SIAM J. Control Optim., 47 (2008), pp.~1078--1095.

\bibitem{CoJiMa12}
{\sc L.~Colombo, F.~Jim{\'e}nez, and D.~Mart{\'{\i}}n~de Diego}, {\em Discrete
  second-order {E}uler-{P}oincar\'e equations. {A}pplications to optimal
  control}, Int. J. Geom. Methods Mod. Phys., 9 (2012), pp.~1250037, 20.

\bibitem{CoMa01}
{\sc J.~Cort{\'e}s and S.~Mart{\'{\i}}nez}, {\em Non-holonomic integrators},
  Nonlinearity, 14 (2001), pp.~1365--1392.

\bibitem{DoHa01}
{\sc A.~L. Dontchev and W.~W. Hager}, {\em The {E}uler approximation in state
  constrained optimal control}, Math. Comp., 70 (2001), pp.~173--203.

\bibitem{DoHa00}
{\sc A.~L. Dontchev, W.~W. Hager, and V.~M. Veliov}, {\em Second-order
  {R}unge-{K}utta approximations in control constrained optimal control}, SIAM
  J. Numer. Anal., 38 (2000), pp.~202--226.

\bibitem{FMOW03}
{\sc R.~C. Fetecau, J.~E. Marsden, M.~Ortiz, and M.~West}, {\em Nonsmooth
  {L}agrangian mechanics and variational collision integrators}, SIAM J. Appl.
  Dyn. Syst., 2 (2003), pp.~381--416.

\bibitem{Gerdts12}
{\sc M.~Gerdts}, {\em Optimal control of {ODE}s and {DAE}s}, de Gruyter
  Textbook, Walter de Gruyter \& Co., Berlin, 2012.

\bibitem{GoRoKaFa08}
{\sc Q.~Gong, I.~M. Ross, W.~Kang, and F.~Fahroo}, {\em Connections between the
  covector mapping theorem and convergence of pseudospectral methods for
  optimal control}, Comput. Optim. Appl., 41 (2008), pp.~307--335.

\bibitem{Ha00}
{\sc W.~W. Hager}, {\em Runge-{K}utta methods in optimal control and the
  transformed adjoint system}, Numer. Math., 87 (2000), pp.~247--282.

\bibitem{Ha01}
\leavevmode\vrule height 2pt depth -1.6pt width 23pt, {\em Numerical analysis
  in optimal control}, in Optimal control of complex structures ({O}berwolfach,
  2000), vol.~139 of Internat. Ser. Numer. Math., Birkh\"auser, Basel, 2002,
  pp.~83--93.

\bibitem{HaLuWa02}
{\sc E.~Hairer, C.~Lubich, and G.~Wanner}, {\em Geometric numerical
  integration}, vol.~31 of Springer Series in Computational Mathematics,
  Springer, Heidelberg, 2010.
\newblock Structure-preserving algorithms for ordinary differential equations,
  Reprint of the second (2006) edition.

\bibitem{HaLe13}
{\sc J.~Hall and M.~Leok}, {\em Spectral variational integrators}.
\newblock arXiv:1211.4534, 2012.

\bibitem{HPS13}
{\sc M.~Herty, L.~Pareschi, and S.~Steffensen}, {\em Implicit-explicit
  {R}unge-{K}utta schemes for numerical discretization of optimal control
  problems}, SIAM J. Numer. Anal., 51 (2013), pp.~1875--1899.

\bibitem{IgMaMa08}
{\sc D.~Iglesias, J.~C. Marrero, D.~Mart{\'{\i}}n~de Diego, and
  E.~Mart{\'{\i}}nez}, {\em Discrete nonholonomic {L}agrangian systems on {L}ie
  groupoids}, J. Nonlinear Sci., 18 (2008), pp.~221--276.

\bibitem{JoMu09}
{\sc E.~R. Johnson and T.~D. Murphey}, {\em Dangers of two-point holonomic
  constraints for variational integrators}, in Proceedings of the 2009
  conference on American Control Conference, ACC'09, Piscataway, NJ, USA, 2009,
  IEEE Press, pp.~4723--4728.

\bibitem{Kane00}
{\sc C.~Kane, J.~E. Marsden, M.~Ortiz, and M.~West}, {\em Variational
  integrators and the {N}ewmark algorithm for conservative and dissipative
  mechanical systems}, Internat. J. Numer. Methods Engrg., 49 (2000),
  pp.~1295--1325.

\bibitem{KoMaSu10}
{\sc M.~Kobilarov, J.~E. Marsden, and G.~S. Sukhatme}, {\em Geometric
  discretization of nonholonomic systems with symmetries}, Discrete Contin.
  Dyn. Syst. Ser. S, 3 (2010), pp.~61--84.

\bibitem{Le04}
{\sc M.~Leok}, {\em Foundations of computational geometric mechanics}, PhD
  thesis, California Institute of Technology, May 2004.
\newblock Thesis (Ph.D.)--California Institute of Technology.

\bibitem{Leok2011}
{\sc M.~Leok and T.~Shingel}, {\em General techniques for constructing
  variational integrators}, Frontiers of Mathematics in China, 7 (2012),
  pp.~273--303.

\bibitem{Lew03}
{\sc A.~Lew, J.~E. Marsden, M.~Ortiz, and M.~West}, {\em Asynchronous
  variational integrators}, Arch. Ration. Mech. Anal., 167 (2003), pp.~85--146.

\bibitem{LMOW04}
\leavevmode\vrule height 2pt depth -1.6pt width 23pt, {\em An overview of
  variational integrators}, in Finite Element Methods: 1970's and Beyond,
  Barcelona, 2004, CIMNE, pp.~98--115.

\bibitem{LMOWe04}
\leavevmode\vrule height 2pt depth -1.6pt width 23pt, {\em Variational time
  integrators}, Internat. J. Numer. Methods Engrg., 60 (2004), pp.~153--212.

\bibitem{leyendecker07-2}
{\sc S.~Leyendecker, J.~E. Marsden, and M.~Ortiz}, {\em Variational integrators
  for constrained dynamical systems}, ZAMM Z. Angew. Math. Mech., 88 (2008),
  pp.~677--708.

\bibitem{LO12}
{\sc S.~Leyendecker and S.~Ober-Bl{\"o}baum}, {\em A variational approach to
  multirate integration for constrained systems}, in Multibody dynamics,
  vol.~28 of Comput. Methods Appl. Sci., Springer, Dordrecht, 2013,
  pp.~97--121.

\bibitem{DMOCC}
{\sc S.~Leyendecker, S.~Ober-Bl{\"o}baum, J.~E. Marsden, and M.~Ortiz}, {\em
  Discrete mechanics and optimal control for constrained systems}, Optimal
  Control Appl. Methods, 31 (2010), pp.~505--528.

\bibitem{Ma92}
{\sc R.~S. MacKay}, {\em Some aspects of the dynamics and numerics of
  {H}amiltonian systems}, in The dynamics of numerics and the numerics of
  dynamics ({B}ristol, 1990), vol.~34 of Inst. Math. Appl. Conf. Ser. New Ser.,
  Oxford Univ. Press, New York, 1992, pp.~137--193.

\bibitem{MPS98}
{\sc J.~E. Marsden, G.~W. Patrick, and S.~Shkoller}, {\em Multisymplectic
  geometry, variational integrators, and nonlinear {PDE}s}, Comm. Math. Phys.,
  199 (1998), pp.~351--395.

\bibitem{MaWe01}
{\sc J.~E. Marsden and M.~West}, {\em Discrete mechanics and variational
  integrators}, Acta Numer., 10 (2001), pp.~357--514.

\bibitem{MaOs04}
{\sc H.~Maurer and N.~P. Osmolovskii}, {\em Second order sufficient conditions
  for time-optimal bang-bang control}, SIAM J. Control Optim., 42 (2004),
  pp.~2239--2263.

\bibitem{O14}
{\sc S.~Ober-Bl\"obaum}, {\em Galerkin variational integrators and modified
  symplectic {R}unge-{K}utta methods}.
\newblock Submitted, 2014.

\bibitem{ObJuMa10}
{\sc S.~Ober-Bl{\"o}baum, O.~Junge, and J.~Marsden}, {\em Discrete mechanics
  and optimal control: an analysis}, ESAIM Control Optim. Calc. Var., 17
  (2011), pp.~322--352.

\bibitem{OS14}
{\sc S.~Ober-Bl\"obaum and N.~Saake}, {\em Construction and analysis of higher
  order galerkin variational integrators}.
\newblock arXiv:1304.1398, 2014.

\bibitem{OTCOM13}
{\sc S.~Ober-Bl{\"o}baum, M.~Tao, M.~Cheng, H.~Owhadi, and J.~E. Marsden}, {\em
  Variational integrators for electric circuits}, J. Comput. Phys., 242 (2013),
  pp.~498--530.

\bibitem{Ro05}
{\sc I.~M. Ross}, {\em A roadmap for optimal control: the right way to
  commute}, Ann. N.Y. Acad., 1065 (2005), pp.~210--231.

\bibitem{RoFa03}
{\sc I.~M. Ross and F.~Fahroo}, {\em Legendre pseudospectral approximations of
  optimal control problems}, in New trends in nonlinear dynamics and control,
  and their applications, vol.~295 of Lecture Notes in Control and Inform.
  Sci., Springer, Berlin, 2003, pp.~327--342.

\bibitem{SS14}
{\sc J.~M. Sanz-Serna}, {\em Symplectic {R}unge-{K}utta schemes for adjoint
  equations, automatic differentiation, optimal control and more}.
\newblock Submitted, 2014.

\bibitem{SaCa94}
{\sc J.~M. Sanz-Serna and M.~P. Calvo}, {\em Numerical {H}amiltonian problems},
  vol.~7 of Applied Mathematics and Mathematical Computation, Chapman \& Hall,
  London, 1994.

\bibitem{StGr09}
{\sc A.~Stern and E.~Grinspun}, {\em Implicit-explicit variational integration
  of highly oscillatory problems}, Multiscale Model. Simul., 7 (2009),
  pp.~1779--1794.

\bibitem{Su90}
{\sc Y.~B. Suris}, {\em Hamiltonian methods of {R}unge-{K}utta type and their
  variational interpretation}, Mat. Model., 2 (1990), pp.~78--87.

\bibitem{TaOwMa2010}
{\sc M.~Tao, H.~Owhadi, and J.~E. Marsden}, {\em Nonintrusive and structure
  preserving multiscale integration of stiff {ODE}s, {SDE}s, and {H}amiltonian
  systems with hidden slow dynamics via flow averaging}, Multiscale Model.
  Simul., 8 (2010), pp.~1269--1324.

\bibitem{Tr00}
{\sc E.~Tr{\'e}lat}, {\em Some properties of the value function and its level
  sets for affine control systems with quadratic cost}, J. Dynam. Control
  Systems, 6 (2000), pp.~511--541.

\bibitem{Tr05}
\leavevmode\vrule height 2pt depth -1.6pt width 23pt, {\em Contr\^ole optimal},
  Math\'ematiques Concr\`etes. [Concrete Mathematics], Vuibert, Paris, 2005.
\newblock Th{\'e}orie \& applications. [Theory and applications].

\bibitem{Tr12}
\leavevmode\vrule height 2pt depth -1.6pt width 23pt, {\em Optimal control and
  applications to aerospace: some results and challenges}, J. Optim. Theory
  Appl., 154 (2012), pp.~713--758.

\bibitem{Wa07}
{\sc A.~Walther}, {\em Automatic differentiation of explicit {R}unge-{K}utta
  methods for optimal control}, Comput. Optim. Appl., 36 (2007), pp.~83--108.

\bibitem{Zu05}
{\sc E.~Zuazua}, {\em Propagation, observation, and control of waves
  approximated by finite difference methods}, SIAM Rev., 47 (2005),
  pp.~197--243 (electronic).

\end{thebibliography}
\end{document}